\documentclass[a4paper,reqno,twoside]{amsart}
\usepackage[english]{babel}
\usepackage[T1]{fontenc}
\usepackage[utf8]{inputenc}
\usepackage{microtype}
\usepackage{stmaryrd}
\usepackage[dvipsnames]{xcolor}
\usepackage{subfiles}

\usepackage[textsize = footnotesize, colorinlistoftodos]{todonotes}

\usepackage{eucal}
\usepackage[shortlabels]{enumitem}
\usepackage[super]{nth}

\usepackage{amsmath}
\usepackage{amssymb}

\usepackage{amsthm}
\usepackage{amsfonts}
\usepackage{graphicx}

\usepackage{dsfont} 
\usepackage{mathtools}
\usepackage{xparse}
\usepackage{tikz-cd}
\usepackage{quiver}
\usepackage{nccmath}

\usepackage{tikz}
\usepackage{tabularx}
\newcolumntype{Y}{>{\centering\arraybackslash}X}
\newcommand{\nwidth}{4cm}
\tikzstyle{no} = [rectangle, rounded corners, 
minimum width=\nwidth, 
minimum height=1cm,
text width = \nwidth,
align=center,
draw=black]

\tikzstyle{arrowv} = [implies-,double equal sign distance,shorten >=0.15cm,shorten <=0.15cm]
\tikzstyle{arrow} = [implies-,double equal sign distance,shorten >=0.05cm,shorten <=0.05cm]

\usepackage[pdfborder={0 0 0}]{hyperref}
\hypersetup{
	colorlinks=true,
	allcolors=.,
	bookmarksdepth = 3
}
\makeatletter
\def\@tocline#1#2#3#4#5#6#7{\relax
  \ifnum #1>\c@tocdepth
  \else
    \par \addpenalty\@secpenalty\addvspace{#2}%
    \begingroup \hyphenpenalty\@M
    \@ifempty{#4}{%
      \@tempdima\csname r@tocindent\number#1\endcsname\relax
    }{%
      \@tempdima#4\relax
    }%
    \parindent\z@ \leftskip#3\relax \advance\leftskip\@tempdima\relax
    \rightskip\@pnumwidth plus4em \parfillskip-\@pnumwidth
    #5\leavevmode\hskip-\@tempdima #6\nobreak\relax
    \leaders\hbox{$\m@th\mkern4.5mu\hbox{.}\mkern4.5mu$}\hfil
    \hbox to\@pnumwidth{\@tocpagenum{#7}}\par
    \nobreak
    \endgroup
  \fi}
\renewcommand\l@subsection{\@tocline{2}{0pt}{2.5pc}{5pc}{}}
\makeatother
\usepackage[capitalize,nameinlink,nosort]{cleveref}
\crefformat{equation}{#2(#1)#3}
\Crefformat{equation}{#2(#1)#3}
\crefname{enumi}{}{}
\crefname{construction}{Construction}{Constructions}

\usepackage{aliascnt}
\let\oldtheorem\newtheorem
\RenewDocumentCommand{\newtheorem}{s m o m O{}}{%
\IfBooleanTF{#1}%
{\oldtheorem{#2}{#4}}%
{\IfNoValueTF{#3}{\oldtheorem{#2}{#4}[#5]}%
{\newaliascnt{#2}{#3}%
\oldtheorem{#2}[#2]{#4}%
\aliascntresetthe{#2}}}}

\theoremstyle{plain}
\newtheorem{theorem}{Theorem}[section]
\newtheorem{lemma}[theorem]{Lemma}
\newtheorem{proposition}[theorem]{Proposition}
\newtheorem{corollary}[theorem]{Corollary}

\newtheorem*{theorem*}{Theorem}
\newtheorem{thmint}{Theorem}
\newtheorem{corint}[thmint]{Corollary}

\theoremstyle{definition}
\newtheorem{definition}[theorem]{Definition}
\newtheorem{remark}[theorem]{Remark}

\newtheorem{notation}[theorem]{Notation}
\newtheorem*{notation*}{Notation}
\newtheorem{construction}[theorem]{Construction}

\makeatletter
\usepackage[normalem]{ulem}
\usepackage{contour}

\contourlength{0.5pt} 

\makeatletter
\DeclareRobustCommand{\myuline}[2][0pt]{%
  \ifmmode
    \uline{\hphantom{#2}\kern-#1}%
    \kern#1%
    \mathllap{\mathpalette\my@cont@{#2}}%
  \else
    \uline{\phantom{#2}\kern-#1}%
    \kern#1%
    \llap{\contour{background}{#2}}%
  \fi
}
\newcommand{\my@cont@}[2]{\contour{background}{\mbox{$\m@th#1#2$}}}
\makeatother

\usepackage{relsize}
\usepackage[bbgreekl]{mathbbol}
\usepackage{amsfonts}
\DeclareSymbolFontAlphabet{\mathbb}{AMSb} 
\DeclareSymbolFontAlphabet{\mathbbl}{bbold}

\newcommand{\id}{\mathrm{id}}
\DeclareMathOperator*{\colim}{colim}

\DeclareMathOperator{\ev}{ev}

\newcommand{\BFree}{\mathsf{BFree}}

\DeclareMathOperator{\Map}{Map}

\DeclareMathOperator{\Ar}{Ar}

\newcommand{\hide}[1]{}

\newlist{numberenum}{enumerate}{1}
\setlist[numberenum]{\upshape(\arabic*)}

\newcommand{\Fun}{\mathrm{Fun}}

\newcommand{\Cat}{\mathcal{C}\mathrm{at}}

\newcommand{\Fib}{\mathcal{F}\mathrm{ib}}
\newcommand{\PFree}{\mathsf{PFree}}

\newcommand{\op}{\mathrm{op}}
\newcommand{\coop}{\mathrm{coop}}
\newcommand{\co}{\mathrm{co}}

\newcommand{\Adj}{\mathsf{Adj}}
\newcommand{\Adjffr}{\mathsf{Adj}^{\operatorname{R}}_{\mathrm{ff}}}
\newcommand{\Adjffl}{\mathsf{Adj}^{\operatorname{L}}_{\mathrm{ff}}}

\newcommand{\Ret}{\mathsf{Ret}}

\newcommand{\soft}[1]{[\![#1]\!]}

\newcommand{\loc}{\mathrm{loc}}

\newcommand{\ddelta}{{\mathlarger{\mathbbl{\Delta}}}}

\mathchardef\mhyphen="2D

\newcommand{\bifib}{\mathrm{bifib}}
\newcommand{\lax}{\mathrm{lax}}
\newcommand{\oplax}{\mathrm{oplax}}

\newcommand{\cell}{\mathsf{cell}}
\newcommand{\intv}{\mathsf{int}}

\newcommand{\ra}{\mathrm{ra}}

\newcommand{\cA}{\mathcal{A}}
\newcommand{\cB}{\mathcal{B}}
\newcommand{\cC}{\mathcal{C}}
\newcommand{\cD}{\mathcal{D}}

\newcommand{\cI}{\mathcal{I}}

\newcommand{\cK}{\mathcal{K}}
\newcommand{\cL}{\mathcal{L}}

\newcommand{\cQ}{\mathcal{Q}}
\newcommand{\cR}{\mathcal{R}}
\newcommand{\cS}{\mathcal{S}}
\newcommand{\cT}{\mathcal{T}}

\newcommand{\cV}{\mathcal{V}}
\newcommand{\cW}{\mathcal{W}}
\newcommand{\cX}{\mathcal{X}}
\newcommand{\cY}{\mathcal{Y}}

\newcommand{\sX}{\mathcal{X}}

\makeatletter
\newcommand{\oset}[3][0ex]{%
	\mathrel{\mathop{#3}\limits^{
			\vbox to#1{\kern-2\ex@
				\hbox{$\scriptstyle#2$}\vss}}}}
\makeatother

\tikzset{
	rot90/.style={anchor=south, rotate=90, inner sep=.5mm}
} 

\renewcommand{\subset}{\subseteq}

\renewcommand{\epsilon}{\varepsilon}

\def\lra{\longrightarrow}
\def\lla{\longleftarrow}

\def\llra{\def\arraystretch{.1}\begin{array}{c} \lra \\ \lla \end{array}}

\makeatletter
\newcommand{\fixed@sra}{$\vrule height 2\fontdimen22\textfont2 width 0pt\rightarrow$}
\newcommand{\shortarrowup}[1]{%
  \mathrel{\text{\rotatebox[origin=c]{65}{\fixed@sra}}}
}
\newcommand{\shortarrowdown}[1]{%
  \mathrel{\text{\rotatebox[origin=c]{250}{\fixed@sra}}}
}
\makeatother

\newcommand{\upslash}{\!\shortarrowup{1}}

\usepackage{silence}
\let\amssection\section
\newcommand{\authorrunhead}{Fernando Abellán}
\uppercasenonmath\authorrunhead
\newcommand{\setsectionrunhead}[1]{%
	\def\sectionrunhead{\thesection.\ #1}%
	\uppercasenonmath\sectionrunhead
	\markboth{\authorrunhead}{\sectionrunhead}%
}
\RenewDocumentCommand{\section}{s o m}{%
	\IfBooleanTF{#1}
		{\amssection*{#3}}
		{\IfNoValueTF{#2}
			{\amssection{#3}\setsectionrunhead{#3}}
			{\amssection[#2]{#3}\setsectionrunhead{#2}}}%
}

\begin{document}
\title{Free bifibrations of $(\infty,2)$-categories, 2-simplicial objects and the walking adjunction}

\author{Fernando Abellán}
\address{Max Planck Institute for Mathematics, Vivatsgasse 7, 53111 Bonn, Germany}
\email{abellan@mpim-bonn.mpg.de}

\begin{abstract}
	In this work, we develop a fibrational approach to freely adjoining adjoints in an $(\infty,2)$-category. We construct the universal bifibration obtained from a cocartesian fibration of $(\infty,2)$-categories by adjoining cartesian lifts over a chosen class of $1$-morphisms in the base. We then use this construction to provide an explicit model for freely adjoining adjoints, together with a zig-zag formula for the resulting mapping $(\infty,1)$-categories. As applications, we give a model-independent proof of the universal property of the walking adjunction, providing an alternative proof of a theorem of Riehl--Verity, and establish the universal characterization of the simplex $2$-category conjectured by Dyckerhoff--Kapranov--Schechtman--Soibelman.
\end{abstract}

\maketitle
\thispagestyle{empty}
\pagestyle{headings}

\tableofcontents

\newpage
\section{Introduction}
Adjunctions are a central organizing principle in higher category theory. Beyond their practical utility, dualizable objects in monoidal $(\infty,n)$-categories and higher adjoints are intimately related to the geometry of manifolds through higher categories of cobordisms. This relationship lies at the heart of the Cobordism Hypothesis of Baez and Dolan \cite{BaezDolan1995HigherDimensionalAlgebra}, further developed in Lurie's work on extended topological field theories \cite{Lurie2009ClassificationTFT}. It is thus reasonable to expect that developing a toolbox for dealing systematically with adjunction data will prove useful while exploring such topics.

A basic question is then how adjoints can be freely adjoined to a collection of $1$-morphisms while retaining control of the resulting higher-categorical structure. In this work, we give an explicit procedure for adjoining such adjoints to $(\infty,2)$-categories using the theory of (co)cartesian\footnote{We later refer to cocartesian fibrations as $(0,1)$-fibrations, following the terminology of \cite{AbellanGagnaea2025StraighteningLaxTransformations,AHM26}.} fibrations.

Let $F \colon \cC \to \Cat_{(\infty,2)}$ be a functor of $(\infty,2)$-categories, and let $p \colon \cX \to \cC$ be the cocartesian fibration corresponding to $F$ under straightening. Given a wide locally full\footnote{Such a sub-$(\infty,2)$-category is specified by its $1$-morphisms.} sub-$(\infty,2)$-category $\cL \subset \cC$, the functor $F$ sends every $1$-morphism in $\cL$ to a left adjoint precisely when $p$ also admits cartesian lifts over $\cL$. We call such a fibration an \emph{$\cL$-bifibration}. Accordingly, we write $\Fib^{(0,1)}_{/\cC}$ for the $(\infty,3)$-category of cocartesian fibrations over $\cC$, and $\Fib_{/(\cC,\cL)}^{(0,1)\mhyphen\bifib}$ for the corresponding $(\infty,3)$-category of $\cL$-bifibrations, whose morphisms additionally preserve the cartesian lifts over $\cL$. Our first main result says that these additional cartesian lifts can be adjoined freely.

\begin{thmint}[{\cref{thm:bifib}}]\label{thm:intro-freebifibration}
	There is an adjunction of $(\infty,3)$-categories
	\[
		\BFree^{(0,1)}_{\cL}\colon
		\Fib^{(0,1)}_{/\cC}
		\llra
		\Fib_{/(\cC,\cL)}^{(0,1)\mhyphen\bifib}
		\colon U,
	\]
	where $U$ is the forgetful functor. We refer to its left adjoint $\BFree^{(0,1)}_{\cL}$ as the \emph{free $\cL$-bifibration} functor.
\end{thmint}


The construction also has a dual, cartesian-variance form, denoted $\BFree^{(1,0)}_{\cL}$. We write $\Fib_{/(\cC,\cL)}^{(1,0)\mhyphen\bifib}$ for the corresponding $(\infty,3)$-category of $\cL$-bifibrations over $\cC$.

Let $\cC\soft{\cL}$ denote the universal $(\infty,2)$-category obtained from $\cC$ by adjoining adjoints to the morphisms of $\cL$. This object exists for formal reasons and comes equipped with a canonical \emph{essentially surjective} functor $\iota\colon\cC\to\cC\soft{\cL}$. Its universal property is expressed by a natural equivalence
\[
	\Fun(\cC\soft{\cL},\cY)
	\xrightarrow{\simeq}
	\Fun(\cC,\cY)_{\cL}
\]
for every $(\infty,2)$-category $\cY$. Here the right-hand side denotes the locally full sub-$(\infty,2)$-category whose objects are the functors carrying the morphisms of $\cL$ to left adjoints and whose $1$-morphisms are the natural transformations with adjointable naturality squares over $\cL$. Thus $\cC\soft{\cL}$ classifies functors out of $\cC$ together with the adjunction data prescribed by $\cL$.

For $c\in\cC$, let $\cC_{\upslash c}\to\cC$ denote the cartesian fibration corresponding to the representable functor $\cC(-,c)$. The free bifibration construction gives the following explicit description of the adjoint completion.

\begin{thmint}[{\cref{thm:addingadjoints}}]\label{thm:intro-adjointcompletion}
	The $(\infty,2)$-category $\cC\soft{\cL}$ is equivalent to the full sub-$(\infty,2)$-category of $\Fib_{/(\cC,\cL)}^{(1,0)\mhyphen\bifib}$ spanned by the free $\cL$-bifibrations
	\[
		\BFree^{(1,0)}_{\cL}(\cC_{\upslash c})\longrightarrow\cC,
		\qquad c\in\cC.
	\]
	In particular, for $x,y \in \cC$, we obtain an equivalence of mapping $(\infty,1)$-categories
		\[
			\cC\soft{\cL}(x,y)
			\simeq
			\BFree^{(1,0)}_{\cL}(\cC_{\upslash y})\times_{\cC}\{x\}.
		\]
\end{thmint}

A crucial feature of our construction is that the fibres of $\BFree^{(1,0)}_{\cL}(\cC_{\upslash c})$ can be described explicitly (see \cref{prop:fibresasexpected}). In \cref{subsec:mappingformula}, we unravel this formula when $\cC$ is an $(\infty,1)$-category.\footnote{We restrict to this case for practical reasons: it keeps the discussion concise and is sufficient for our applications.} For each $n\geq0$, let $\cC^{(n)}(x,y)$ denote the $(\infty,1)$-category of alternating zig-zags of length $n+1$ from $x$ to $y$ whose right-pointing morphisms belong to $\cL$, with morphisms given by maps of zig-zags. These assemble into a filtered diagram, and we write
\[
	\cC^{\infty}(x,y)=\colim_{n\in\mathbb N}\cC^{(n)}(x,y).
\]
There are natural classes $\cQ_{n,x}$ of maps of zig-zags, determining a class $\cQ_x$ in $\cC^{\infty}(x,y)$, which we describe in \cref{def:QnLocalizefibre}. Inverting these classes, we obtain the following formula for the mapping $(\infty,1)$-categories of the adjoint completion.

\begin{thmint}[{\cref{thm:mappingcatform}}]\label{thm:intro-mappingformula}
	Let $\cL\subset\cC$ be a wide sub-$(\infty,1)$-category. For every $x,y\in\cC$, there are natural equivalences of $(\infty,1)$-categories
	\[
		\cC\soft{\cL}(x,y)
		\simeq
		\colim_{n\in\mathbb N}\cC^{(n)}(x,y)[\cQ_{n,x}^{-1}]
		\simeq
		\cC^{\infty}(x,y)[\cQ_x^{-1}].
	\]
\end{thmint}

Observe that localizing $\cC\soft{\cL}$ at all its $2$-morphisms yields $\cC[\cL^{-1}]$. By \cref{thm:intro-adjointcompletion}, this amounts to geometrically realizing the fibres of the corresponding free bifibrations. As a direct consequence, we obtain a description of the mapping spaces of $\cC[\cL^{-1}]$ which closely resembles the hammock localization formula of Dwyer and Kan \cite{DwyerKan1980Calculating}.

\begin{corint}[{\cref{rem:softstrongloc,cor:dwyerkanhammock}}]\label{cor:intro-hammock}
	Let $\cL\subset\cC$ be a wide sub-$(\infty,1)$-category. For every $x,y\in\cC$, there are equivalences of spaces
	\[
		\cC[\cL^{-1}](x,y)
		\simeq
		\colim_{n\in\mathbb N}\lvert\cC^{(n)}(x,y)\rvert
		\simeq
		\lvert\cC^{\infty}(x,y)\rvert.
	\]
\end{corint}
\subsection*{Applications}
We now turn to two applications in which our description of $\cC\soft{\cL}$ can be made completely explicit.

Our first application concerns the walking adjunction. Riehl and Verity proved that every adjunction in a suitably fibrant simplicially enriched category extends to a homotopy-coherent adjunction \cite{RiehlVerity2016FormalTheoryMonads}. In the language used here, their result says that the classical $2$-category $\Adj$ represents left adjoints in any $(\infty,2)$-category. Our construction gives a model-independent proof of this universal property.

\begin{thmint}[{\cref{thm:RVmi}}]\label{thm:intro-walking-adjunction}
	Let $\Adj$ denote the walking adjunction, and let $[1]\to\Adj$ select its universal left adjoint. For every $(\infty,2)$-category $\cX$, restriction induces an equivalence of $(\infty,2)$-categories
	\[
		\Fun(\Adj,\cX)\xrightarrow{\simeq}\Fun([1],\cX)_{\mathrm{LAdj}},
	\]
	where the right-hand side denotes the locally full sub-$(\infty,2)$-category of $\Fun([1],\cX)$ whose objects are the left adjoints in $\cX$ and whose $1$-morphisms are the natural transformations with adjointable naturality squares.
\end{thmint}

Our second application concerns the simplex $2$-category $\ddelta$. This is the full sub-$(\infty,2)$-category of $\Cat_{(\infty,1)}$ spanned by the nonempty finite ordinals, with $2$-morphisms given by pointwise inequalities between order-preserving maps.

The simplex $2$-category features in the categorified Dold--Kan correspondence of Dyckerhoff \cite[Theorem 3.4]{Dyckerhoff2021CategorifiedDoldKan}, which establishes an equivalence between the $(\infty,1)$-categories of connective chain complexes of stable $(\infty,1)$-categories and $2$-simplicial stable $(\infty,1)$-categories, i.e. functors $\cX_\bullet \colon \ddelta^\op \to \cS\mathsf{t}$\footnote{We denote by $\cS\mathsf{t}$ the $(\infty,2)$-category of stable $(\infty,1)$-categories and exact functors.}. Fundamental examples of such $2$-simplicial objects arise from variants of Waldhausen's $S_\bullet$-construction in algebraic $K$-theory.

A key property of $\ddelta$ is that its face and degeneracy maps assemble into adjunction strings
\[
	d_{i+1}\dashv s_i\dashv d_i.
\]

Dyckerhoff, Kapranov, Schechtman, and Soibelman conjectured that these adjunctions give a universal characterization of $\ddelta$ \cite[Conjecture~3.3.7]{DyckerhoffKapranovSchechtmanSoibelman2024SphericalAdjunctions}. Let $\Delta$ be the usual simplex $1$-category and write $i \colon \Delta \to \ddelta$ for the canonical inclusion.

\begin{thmint}[{\cref{thm:2simplex}}]\label{thm:intro-simplex}
	Let $\cX$ be an $(\infty,2)$-category. Then restriction along $i\colon \Delta \to \ddelta$ induces a fully faithful functor of $(\infty,2)$-categories
	\[
		i^* \colon \Fun(\ddelta,\cX) \longrightarrow \Fun(\Delta,\cX).
	\]
	Its essential image consists of those cosimplicial objects $F \colon \Delta \to \cX$ that send each invertible $2$-morphism $\id \xrightarrow{\simeq} s_i d_{i+1}$ to the unit and each $s_i d_i \xrightarrow{\simeq} \id$ to the counit of an adjunction.
\end{thmint}

The preceding theorem will also play an important role in forthcoming joint work with Thomas Blom \cite{AbellanBlomLaxIdempotentMonads,AbellanBlomNonCoherentLaxIdempotency}, where we develop the theory of lax-idempotent monads---and, more generally, lax-idempotent algebras---in a homotopy-coherent setting.

\subsection*{Relation to the literature and folklore}
The idea of using bifibrations in order to encode adjunctions can be traced back to the school of Grothendieck under the name of \emph{bifibred} categories \cite[Exposé~VI, \S~10]{Grothendieck1971SGA1}. This point of view is also closely related to the use of fibred categories in monadic descent, going back to Bénabou and Roubaud \cite{BenabouRoubaud1970MonadesDescente}.

The particular alternating construction for the free bifibration used in this paper appears to have circulated as folklore: I learned from Thomas Blom that the free bifibration should be obtained by alternately adjoining cartesian and cocartesian lifts. According to Blom, he learned about this construction from Germán Stefanich, who had himself heard about it from Nick Rozenblyum.\footnote{It is unclear whether the lineage of this construction goes even deeper.}

Folklore aside, there is substantial literature on closely related free constructions. Dawson, Paré, and Pronk construct the $2$-category obtained from a category by freely adjoining right adjoints to all its morphisms \cite{DawsonParePronk2003AdjoiningAdjoints}. Recently, a strict categorical construction of the free bifibration was given by Clarke, Scherer, and Zeilberger \cite{ClarkeSchererZeilberger2025FreeBifibration}. At the higher-categorical level, Riva and Rovelli construct another model for freely adjoining right adjoints to an $(\infty,1)$-category using a particular zig-zag construction \cite{RivaRovelli2025ZigzagsAdjunctions}. To the best of our understanding, none of the existing literature provides sufficient control over the mapping $(\infty,1)$-categories needed for our applications; this justifies adding yet another paper on this classical question.
\subsection*{Use of Large Language Models}
Throughout the development of this paper, large language models (LLMs) were used as editorial tools to detect mathematical errors, typos and inconsistencies in notation, as well as to conduct advanced literature searches. LLMs were also used in a series of experiments exploring the combinatorics of the category $Q(n,n)_{\id}^{(k)}$ from \cref{not:Qnn}. More precisely, for a family of subcategories of $Q(n,n)_{\id}^{(k)}$, an LLM was instructed to count the morphisms to a given object $I_n^{(k)}$ for small values of $k$. I was particularly pleased to find that, for a given choice of subcategory, these experiments found precisely one morphism with target $I_n^{(k)}$. This observation ultimately led to the proof of \cref{prop:Qcontract}.
\subsection*{Acknowledgements}
I would like to thank Tobias Dyckerhoff for introducing me to the simplex $2$-category conjecture during the early stages of my PhD. Much of my work on the theory of $(\infty,2)$-categories has been indirectly motivated by the desire to develop a general toolbox for tackling questions of this kind. It is also my pleasure to thank Rune Haugseng for his persistent—and occasionally relentless—encouragement, throughout my years at the Norwegian University of Science and Technology, to find a new proof of the universal property of the walking adjunction using modern $(\infty,2)$-categorical technology. I would also like to express my gratitude to Thomas Blom. This project was motivated by our joint work on lax-idempotent algebras, and I am indebted to him for several discussions that greatly influenced this paper. Finally, I thank the Max Planck Institute for Mathematics in Bonn, where this work was carried out, for its hospitality and support.
\section{Preliminaries}
We assume that the reader is familiar with the basic language of higher category theory. The purpose of this section is primarily to fix notation and to recall the few constructions used throughout the paper.

\subsection{Generalities on $(\infty,2)$-categories}

\begin{notation}\label{not:adjointstring}
	\leavevmode\par
	\begin{itemize}
		\item We write $\Cat_{(\infty,n)}$ for the $(\infty,n+1)$-category of $(\infty,n)$-categories and $\Fun(\cC,\cD)$ for the corresponding functor $(\infty,n)$-category.
		\item For an $(\infty,2)$-category $\cC$, we write $\cC(x,y)$ for its mapping $(\infty,1)$-category. We use $\Map_{\cA}(f,g)$ for the mapping space between two objects of an $(\infty,1)$-category $\cA$.
		\item The superscripts $\op$, $\co$, and $\coop$ indicate reversal of 1-morphisms, 2-morphisms, and both, respectively.

		\item Unless otherwise indicated, $\Fun(\cC,\cD)$ is formed using natural transformations; the variants with lax or oplax transformations are denoted by $\Fun^{\lax}(\cC,\cD)$ and $\Fun^{\oplax}(\cC,\cD)$.
		\item We write $\cC^{\simeq}$ for the underlying $\infty$-groupoid (or space) of an $(\infty,2)$-category.  
		\item We write $|\cC|$ for the $\infty$-groupoid obtained by inverting every morphism, and $\cC[\cW^{-1}]$ for the localization of $\cC$ at a specified class $\cW$ of morphisms.
		\item We denote by $\iota \colon \Cat_{(\infty,1)} \hookrightarrow \Cat_{(\infty,2)}$ the canonical inclusion.
		\item We have a string of adjoints $\tau_1 \dashv \iota \dashv (-)^{\leq 1}$ of $(\infty,2)$-categories, where $\tau_1 \cC$ inverts all 2-morphisms and $\cC^{\leq 1}$ discards all non-invertible 2-morphisms. 
	\end{itemize}
\end{notation}

\begin{definition}\label{def:adjointable}
	A commutative square in an $(\infty,2)$-category
	\[
	\begin{tikzcd}
		\cA \arrow[r,"f"] \arrow[d,"u"'] & \cB \arrow[d,"v"] \\
		\cC \arrow[r,"g"'] & \cD
	\end{tikzcd}
	\]
	whose horizontal morphisms admit right adjoints $f^R$ and $g^R$ is
	\emph{adjointable} if the mate transformation
	\[
		u f^R \xrightarrow{} g^R g u f^R
		\xrightarrow{} g^R v f f^R \xrightarrow{} g^R v,
	\]
	is an equivalence, where the first and last maps are induced by the unit of $g \dashv g^R$
	and the counit of $f \dashv f^R$, respectively. If $f,g$ admit left adjoints one defines analogously the property of being an adjointable square.
\end{definition}

\begin{lemma}\label{BCeasy}
	Let $\sX$ be an $(\infty,2)$-category and
	\[\begin{tikzcd}
		x & {x'} \\
		y & {y'}
		\arrow["f",from=1-1, to=1-2]
		\arrow["i"', from=1-1, to=2-1]
		\arrow["{j}", from=1-2, to=2-2]
		\arrow["g"',from=2-1, to=2-2]
	\end{tikzcd}\]
	a commutative square in $\sX$.
	Suppose $i$ and $j$ admit left adjoints $k$ and $l$ which are reflective localizations.
	Then the Beck--Chevalley transformation
	\[\begin{tikzcd}
		x & {x'} \\
		y & {y'}
		\arrow["f", from=1-1, to=1-2]
		\arrow[shorten <= 6, shorten >= 6, Rightarrow, from=2-2, to=1-1]
		\arrow["k", from=2-1, to=1-1]
		\arrow["g"', from=2-1, to=2-2]
		\arrow["l"', from=2-2, to=1-2]
	\end{tikzcd}\]
	is invertible if and only if there exists some homotopy $lg \simeq fk$ making this square commute. 
\end{lemma}

\begin{proof}
	The Beck--Chevalley transformation is the composite
	\[lg \xRightarrow{\eta} lgik \simeq ljfk \xRightarrow{\epsilon} fk.\]
	Since $\varepsilon$ is invertible by assumption, it suffices to show that $\eta \colon lg \Rightarrow lgik$ is invertible.
	Because of the homotopy $lg \simeq fk$, this holds if $\eta \colon fk \Rightarrow fkik$ is invertible.
	Since the 2-morphism $\eta \colon k \Rightarrow kik$ is invertible by the triangle identities, the result follows.
\end{proof}

\begin{remark}\label{rem:BCeasy}
	\cref{BCeasy} admits obvious dual statements where:
	\begin{itemize}
		\item the functors $i,j$ admit right adjoints which are coreflective localizations. 
		\item the functors $i,j$ are themselves (co)reflective localizations.
	\end{itemize} 
\end{remark}

\begin{definition}
	We say that a functor $f \colon \cC \to \cD$ of $(\infty,2)$-categories is:
	\begin{itemize}
		\item fully faithful if it induces an equivalence on mapping $(\infty,1)$-categories.
		\item locally fully faithful if it induces a fully faithful functor on mapping $(\infty,1)$-categories.
		\item $k$-essentially surjective, if it is essentially surjective on $l$-morphisms for $l \leq k$.
	\end{itemize}
\end{definition}

\begin{proposition}\label{prop:essurjcobase}
	The class of $k$-essentially surjective functors is stable under cobase change.
\end{proposition}
\begin{proof}
	This follows from \cite[Theorem B.4.1]{Soerg}.
\end{proof}

\begin{definition}\label{def:typeofsubcat}
	For a functor of $(\infty,2)$-categories $f \colon \cC \to \cD$, we say that:
	\begin{itemize}
		\item $f$ is \emph{wide} if it induces an equivalence on underlying spaces.
		\item $f$ is \emph{locally wide} if for every pair of objects the induced map on mapping $(\infty,1)$-categories is wide.
		\item $f$ is a \emph{wide} sub-$(\infty,2)$-category if it is wide, and for every pair of objects the induced map on mapping categories is a sub-$(\infty,1)$-category inclusion, i.e. it induces a monomorphism on mapping spaces.
		\item $f$ is a wide locally full sub-$(\infty,2)$-category if it is wide and locally fully faithful.
	\end{itemize}
\end{definition}

\subsubsection{Adjunctions}
This subsection introduces the classical $2$-categories $\Adj$, $\Adjffr$, and $\Adjffl$, encoding respectively the data of an adjunction, a fully faithful right adjoint, and a fully faithful left adjoint. In \cref{sec:walkingadjunction} we give a new proof that these classical objects satisfy the corresponding universal property in the higher categorical setting, recovering a theorem originally due to Riehl and Verity \cite{RiehlVerity2016FormalTheoryMonads}.

Concretely, $\Adj$ admits the following combinatorial description, due to Schanuel and Street \cite{SchanuelStreet1986FreeAdjunction}.

\begin{definition}[Schanuel--Street]\label{def:AdjSS}
	Write $\Delta_+$ for the augmented simplex category: the category of finite ordinals $[n] = \{0 < 1 < \cdots < n\}$ for $n\geq 0$, together with the empty ordinal $[-1]=\emptyset$, and order-preserving maps, regarded as a full sub-category of $\Cat$. Let $\Delta_{\infty}, \Delta_{-\infty},\Delta_{\intv} \subset \Delta_+$ denote the subcategories on the nonempty ordinals whose morphisms preserve, respectively, the top element, the bottom element, and both endpoints. The $2$-category $\Adj$ has two objects $-,+$ and mapping categories
	\[
		\begin{aligned}
			\Adj(+,+) &\coloneqq \Delta_+, & \Adj(-,-) &\coloneqq \Delta_{\intv}\simeq (\Delta_+)^{\op}, \\
			\Adj(-,+) &\coloneqq \Delta_{\infty} \cong \Delta_{-\infty}^{\op}, & \Adj(+,-) &\coloneqq \Delta_{-\infty} \cong \Delta_{\infty}^{\op},
		\end{aligned}
	\]
	with composition induced by the ordinal sum bifunctor $\oplus \colon \Delta_+ \times \Delta_+ \to \Delta_+$, restricted and dualized appropriately to each pair of mapping categories. See \cite[Remark 3.3.8]{RiehlVerity2016FormalTheoryMonads} for the complete description.
\end{definition}

\begin{definition}\label{def:adjffr}
	The \emph{walking fully faithful right adjoint} $\Adjffr$ is the $(\infty,2)$-category obtained from $\Adj$ by freely inverting the counit of the universal adjunction. Concretely, $\Adjffr$ is a $(2,2)$-category with two objects $-,+$ whose mapping categories are
	\[
		\Adjffr(+,+) = [1], \qquad \Adjffr(x,y) = [0] \enspace \text{ for } (x,y) \neq (+,+),
	\]
	where $[0]$ is the terminal category and $[1] = (0 \to 1)$ is the walking arrow. The composition rule is determined by endowing $[1]$ with the monoidal structure given by the maximum of two elements, together with the functor
	\[
		\Adjffr(+,-) \times \Adjffr(-,+) \longrightarrow \Adjffr(+,+)
	\]
	which sends the unique object to $1 \in [1]$.

	Dually, the \emph{walking fully faithful left adjoint} $\Adjffl$ is the $(\infty,2)$-category obtained from $\Adj$ by freely inverting the unit $\eta$, and admits a dual description to that of $\Adjffr$.
\end{definition}

\subsection{Fibrations of $(\infty,2)$-categories}
In this section, we recall the notion of (co)cartesian fibrations of $(\infty,2)$-categories. We will assume that the reader is familiar with the $(\infty,1)$-categorical counterpart of the theory (see \cite[Section 2.4]{HTT}).
\begin{definition}
	Let $p \colon \cX \to \cC$ be a functor of $(\infty,2)$-categories. We say that a morphism $u \colon x \to y$ in $\cX$ is \emph{cocartesian} if for every $z \in \cX$ the following diagram
	\[
		\begin{tikzcd}
		\cX(y,z) \arrow[r,"u^*"] \arrow[d]  & \cX(x,z) \arrow[d] \\
		\cC(py,pz) \arrow[r,"pu^*"] & \cC(px,pz)
	\end{tikzcd}
	\]
	is a pullback square. Dually, we will say that $u$ is cartesian if for every $z$ as before, the diagram
	\[
		\begin{tikzcd}
		\cX(z,x) \arrow[r,"u_*"] \arrow[d]  & \cX(z,y) \arrow[d] \\
		\cC(pz,px) \arrow[r,"pu_*"] & \cC(pz,py)
	\end{tikzcd}
	\]
	is a pullback square.
\end{definition}

\begin{definition}
	A functor $p\colon \cX \to \cC$ of $(\infty,2)$-categories is said to be \emph{cocartesian enriched} if for every $x,y \in \cX$ the induced functor on mapping $(\infty,1)$-categories $p_{x,y} \colon \cX(x,y) \to \cC(px,py)$ is a cocartesian fibration and for every triple of objects $x,y,z\in \cX$ the composition functor $\cX(x,y)\times \cX(y,z)\to \cX(x,z)$ preserves cocartesian morphisms. We dually define the notion of a \emph{cartesian enriched} functor.
\end{definition}

\begin{definition}
	A functor $p \colon \cX \to \cC$ of $(\infty,2)$-categories is said to be an $(i,j)$-fibration if
	\begin{itemize}
		\item[$(i)$] It is \emph{cocartesian} enriched if $j=0$ or \emph{cartesian} enriched if $j=1$.
		\item[$(ii)$] Given $i=0$, we have cocartesian lifts of all 1-morphisms; that is, for every $u \colon p(x)\to s$ in $\cC$ there exists a cocartesian morphism $\hat{u}\colon x \to \hat{s}$ with $p(\hat{u})=u$. Dually, if $i=1$, we have cartesian lifts of all 1-morphisms: for every $v \colon x\to p(s)$ in $\cC$ there exists a cartesian morphism $\hat{v}\colon \hat{x} \to s$ with $p(\hat{v})=v$.
	\end{itemize}
	A morphism of $(i,j)$-fibrations is a functor over $\cC$ preserving (co)cartesian 1- and 2-morphisms. For $\varepsilon=(i,j)$ denote the corresponding $(\infty,3)$-category of $\varepsilon$-fibrations as $\Fib^{\epsilon}_{/\cC}$.
\end{definition}

\begin{notation}
	We fix the notation $i$-cartesian to denote either cocartesian for $i=0$ or cartesian for $i=1$.
\end{notation}

\begin{notation}\label{not:Cepsilon}
	For an $(\infty,2)$-category we let $\cC^{\epsilon}$ denote $\cC$ if $\epsilon=(0,1)$, $\cC^{\op}$ if $\epsilon=(1,0)$, $\cC^{\co}$ if $\epsilon=(0,0)$ and $\cC^{\coop}$ for $\epsilon=(1,1)$.
\end{notation}

\begin{theorem}[{\cite{NuiStr,AbellanStern2026TwoCartesianII}}]\label{thm:str}
	There is an equivalence of $(\infty,2)$-categories
	\[
		\mathsf{St}^{\epsilon}_{\cC} \colon \Fib^{\epsilon}_{/\cC} \llra \Fun(\cC^{\epsilon},\Cat_{(\infty,2)})\colon \mathsf{Un}_{\cC}^{\epsilon},
	\]
	where $\cC^{\epsilon}$ is given as in \cref{not:Cepsilon}.
\end{theorem}

It is often convenient to work with a relative variant of the usual notion of a cocartesian fibration, where lifts are only demanded along a chosen wide subcategory of the base, rather than along every morphism.

\begin{definition}\label{def:partialcocart}
	Let $\cC_0 \subset \cC$ be a wide locally full sub-$(\infty,2)$-category. A functor $p \colon \cX \to \cC$ is a \emph{partial} $(i,j)$-fibration relative to $\cC_0$ if it is $j$-enriched and it admits $i$-cartesian lifts for every morphism of $\cC_0$. We write $\Fib^{\epsilon}_{/(\cC,\cC_0)}$ for the corresponding $(\infty,3)$-category of partial $(i,j)$-cartesian fibrations relative to $\cC_0$, whose 1-morphisms are those which preserve the $i$-cartesian morphisms over $\cC_0$ and every $j$-cartesian 2-morphism.
\end{definition}

\begin{notation}\label{not:funpartial}
	For $p \colon \cX \to \cC$ and $q \colon \cY \to \cC$ we will denote $\Fun^{\epsilon}_{/(\cC,\cC_0)}(\cX,\cY)$ the mapping $(\infty,2)$-category in $\Fib^{\epsilon}_{/(\cC,\cC_0)}$. If $\cC=\cC_0$, so that we are dealing with usual $\epsilon$-fibrations, we will use the lighter notation $\Fun^{\epsilon}_{/\cC}(\cX,\cY)$.
\end{notation}

As a direct consequence of \cite[Theorem 3.4.1]{AHM26} we obtain the following result.

\begin{proposition}\label{prop:fibfunct}
	Let $p \colon \cX \to \cC$ be a partial $(i,j)$-fibration relative to $\cC_0$. Then the induced functor 
	\[
		p_* \colon \Fun(\cA,\cX) \xrightarrow{} \Fun(\cA,\cC),
	\]
	is a partial $(i,j)$-fibration, with respect to the wide locally full sub-$(\infty,2)$-category consisting of those natural transformations whose components belong to $\cC_0$.
\end{proposition}

\begin{definition}\label{def:freepartial}
	Let $p \colon \cX \to \cC$ be a functor, $\cC_0 \subset \cC$ a wide locally full sub-$(\infty,2)$-category, and $\epsilon=(i,j)$. Write
	\[
		\Ar^{\epsilon}_{\cC_0}(\cC) \subset \Ar^{\epsilon}(\cC) \coloneqq \begin{cases}
			\Fun^{\oplax}([1],\cC) & \text{if } i \neq j, \\
			\Fun^{\lax}([1],\cC) & \text{if } i = j,
		\end{cases}
	\]
	for the full subcategory spanned by those morphisms lying in $\cC_0$, with evaluation maps $\ev_0,\ev_1 \colon \Ar^{\epsilon}_{\cC_0}(\cC) \to \cC$. The \emph{free partial $\epsilon$-fibration on $p$ relative to $\cC_0$} is the functor
	\[
		\mathsf{PFree}^{\epsilon}_{\cC_0}(p)\colon \mathsf{PFree}^{\epsilon}_{\cC_0}(\cX)=\cX \times_{\cC} \Ar^{\epsilon}_{\cC_0}(\cC) \longrightarrow \Ar^{\epsilon}_{\cC_0}(\cC) \xrightarrow{\ev_{1-i}} \cC,
	\]
	where the fibre product is taken along $\ev_i$.
\end{definition}

\begin{theorem}[{\cite[Theorem 4.2.1, Observation 4.2.8]{AHM26}}]\label{thm:pfree}
	We have adjunctions of $(\infty,3)$-categories
	\[
		\PFree^{\epsilon}_{\cC_0}\colon {\Cat_{(\infty,2)}}_{/\cC} \llra \Fib^{\epsilon}_{/(\cC,\cC_0)}\colon U^{\epsilon}, 
	\]
	where the right adjoint denotes the obvious forgetful functor.
\end{theorem}

\begin{definition}
	Let $\cC$ be an $(\infty,2)$-category and let $c \in \cC$. We denote by $\cC_{\upslash c} \to \cC$ the $(1,0)$-fibration corresponding to the representable functor on $\cC$ via \cref{thm:str}. We observe that this fibration can be identified with the free $(1,0)$-fibration on the functor $[0] \to \cC$ selecting the object $c \in \cC$.
\end{definition}

\subsection{Cofinal maps}
The theory of $(\infty,2)$-categories supports a well-developed theory of (co)limits arising from enriched $\infty$-category theory, usually referred to as the theory of weighted (co)limits \cite{Heine2024HigherAlgebraWeightedColimits}. An alternative formulation in terms of partially (op)lax (co)limits has been developed in \cite{GagnaHarpazLanari2026MarkedLimits,AHM26} and compared with the enriched theory in \cite[Section 5.2]{AHM26}. One advantage of the partially lax approach is that it admits a clean generalization of the notion of a cofinal functor. We will need only a small part of this theory, which we recall here for completeness.

\begin{definition}[{\cite[Definition 5.1.2]{AHM26}}]
	Let $F \colon \cI \to \cA$ be a functor of $(\infty,2)$-categories. We say that $a \in \cA$ is the \emph{strong limit} of $F$ if it represents the functor sending each $x \in \cA$ to the $\infty$-category of natural transformations from the constant functor on $x$ to $F$. We write $\lim_{\cI}F$ for the strong limit of $F$ whenever it exists.
\end{definition}

\begin{definition}[{\cite[Definition 5.5.7 and Proposition 5.5.8]{AHM26}}]\label{def:strongcofinal}
	A functor $f \colon \cC \to \cD$ is \emph{strongly $(0,1)$-cofinal} if, for every $(\infty,2)$-category $\cA$ and every functor $F \colon \cD \to \cA$, the limit of $F$ exists if and only if the limit of $F\circ f$ exists, and the canonical map
	\[
		\lim_{\cD}F \xrightarrow{\simeq} \lim_{\cC}(F\circ f)
	\]
	is an equivalence.
\end{definition}

\begin{proposition}[{ \cite[Theorem 5.5.13]{AHM26}}]\label{prop:strongcofinalcharact}
	Let $f \colon \cC \to \cD$ be a functor of $(\infty,2)$-categories. Then $f$ is strongly $(0,1)$-cofinal if and only if, for every $d \in \cD$, the $(\infty,2)$-category
	\[
		\cC_{\upslash d}\coloneq \cD_{\upslash d}\times_{\cD}\cC
	\]
	becomes equivalent to $[0]$ after localizing at all 2-morphisms and at those 1-morphisms represented by commutative triangles.
\end{proposition}

\begin{proposition}[{\cite[Observation 5.5.5.]{AHM26}}]\label{prop:cofcobase}
	The class of strong $(0,1)$-cofinal functors is stable under cobase change.
\end{proposition}

\begin{proposition}\label{prop:adjffcofinal}
	 The functors $\Ret \to \Adjffr$ and $\Ret \to \Adjffl$ are strongly $(0,1)$-cofinal where $\Ret=[2]\coprod_{\{0 \to 2\}}[0]$ denotes the walking retraction.
\end{proposition}
\begin{proof}
	We treat the functor $\Ret \to \Adjffr$; the other case is completely dual. By \cref{prop:strongcofinalcharact}, it is enough to show that, for every $x \in \Adjffr$, the $(\infty,1)$-category
	\[
		\Ret_{\upslash x}\coloneq \Ret \times_{\Adjffr}{\Adjffr}_{\upslash x}
	\]
	localizes to $[0]$ after inverting the 1-morphisms represented by commutative triangles. For $x \in \Adjffr$ denote by $\rho_x \colon - \to x$ the unique morphism. Then it follows that $\rho_x \in \Ret_{\upslash x}$ yields an initial object and that every morphism with source $\rho_x$ is represented by a commutative triangle. Recall the well known fact that localizations are obtained by iterated pushouts along the map $[1] \to [0]$, which is easily verified to be strong $(0,1)$-cofinal. In particular, one concludes from \cref{prop:cofcobase} that localizations preserve initial objects. The claim now follows immediately.
\end{proof}

\section{Free bifibrations of $(\infty,2)$-categories}

\begin{definition}\label{def:wcart}
	Let $p\colon\cX\to\cC$ be an $(i,j)$-fibration and let $\cL\subset\cC$ be a wide locally full sub-$(\infty,2)$-category. We call $p$ an \emph{$\cL$-bifibration} if it is also a partial $(1-i,j)$-fibration relative to $\cL$. We write
	\[
		\Fib_{/(\cC,\cL)}^{\epsilon\mhyphen\bifib}
		\subset \Fib_{/\cC}^{\epsilon}
	\]
	for the resulting $(\infty,3)$-category, whose morphisms are the morphisms of $(i,j)$-fibrations that also preserve the $(1-i)$-cartesian morphisms lying over $\cL$.
\end{definition}

We write $\Fun_{/\cC}^{\cL\mhyphen\bifib}(\cX,\cY)$ for the corresponding mapping $(\infty,2)$-category. When $\cL=\cC$, we use the lighter notation $\Fib_{/\cC}^{\epsilon\mhyphen\bifib}$ and simply speak of bifibrations.

The main result of this section is the following.

\begin{theorem}\label{thm:bifib}
	There is an adjunction of $(\infty,3)$-categories
	\[
		\BFree^{\epsilon}_{\cL}\colon
		\Fib^{\epsilon}_{/\cC}
		\llra
		\Fib_{/(\cC,\cL)}^{\epsilon\mhyphen\bifib}
		\colon U,
	\]
	where $U$ is the forgetful functor.
\end{theorem}

 We refer to the left adjoint as the \emph{free bifibration} functor. Before explaining our strategy for proving \cref{thm:bifib}, we begin with a useful characterization of bifibrations.
 
\begin{proposition}\label{prop:charactbifib}
	Let $p\colon\cX\to\cC$ be an $(i,j)$-fibration. Then $p$ is an $\cL$-bifibration precisely when the following condition holds:
	\begin{itemize}
		\item if $i=0$, then for every $u\colon x\to y$ in $\cL$, the transport functor $u_!\colon\cX_x\to\cX_y$ is a left adjoint;
		\item if $i=1$, then for every $v\colon y\to z$ in $\cL$, the transport functor $v^*\colon\cX_z\to\cX_y$ is a right adjoint.
	\end{itemize}
\end{proposition}
\begin{proof}
	If $p$ is an $\cL$-bifibration, the claim follows from \cite[Corollary 4.1.3]{AbellanGagnaea2025StraighteningLaxTransformations}. Conversely, we may assume that $\epsilon=(0,1)$. The same result gives locally cartesian lifts over the morphisms of $\cL$; it remains to show that they are cartesian.

	Let $\widehat{v}\colon x\to y$ be locally cartesian and let $p(\widehat{v})=v$. For every $z\in\cX$, we must show that
	\[
	\begin{tikzcd}
		\cX(z,x) \arrow[r,"\widehat{v}_*"] \arrow[d] & \cX(z,y) \arrow[d] \\
		\cC(pz,px) \arrow[r,"v_*"'] & \cC(pz,py)
	\end{tikzcd}
	\]
	is a pullback. Since the vertical maps are cartesian fibrations, this may be checked fibrewise. Given $u\in\cC(pz,px)$, choose a cocartesian lift $\widehat{u}\colon z\to\widehat x$. We obtain a commutative square
	\[
	\begin{tikzcd}
		\cX(\widehat x,x)_{\id_{px}} \arrow[r,"\widehat u^*"] \arrow[d,"\widehat v_*"']
		& \cX(z,x)_u \arrow[d,"\widehat{v}_*"] \\
		\cX(\widehat x,y)_{v} \arrow[r,"\widehat u^*"']
		& \cX(z,y)_{v\circ u}.
	\end{tikzcd}
	\]
	The horizontal maps are equivalences because $\widehat u$ is cocartesian, and the left vertical map is an equivalence because $\widehat v$ is locally cartesian. Hence the right vertical map is an equivalence by 2-out-of-3.
\end{proof}

\begin{construction}\label{cons:Tbifib}
	Let $p\colon\cX\to\cC$ be an $(i,j)$-fibration. We define functors $p^{(n)}\colon\cX^{(n)}\to\cC$ inductively. Set $p^{(0)}=p$. If $n$ is odd, let $p^{(n)}$ be the free partial $(1-i,j)$-fibration relative to $\cL$ on $p^{(n-1)}$; if $n>0$ is even, let $p^{(n)}$ be the free $(i,j)$-fibration on $p^{(n-1)}$.

	The units of the corresponding adjunctions assemble these functors into a diagram
	\[
		T\colon\mathbb N\longrightarrow {\Cat_{(\infty,2)}}_{/\cC}.
	\]
	We write $p_T\colon\cX_T\to\cC$ for its colimit.
\end{construction}

\begin{notation}\label{not:Tkell}
	For $k\leq\ell$, write
	\[
		T_{k,\ell}\colon\cX^{(k)}\longrightarrow\cX^{(\ell)}
	\]
	for the transition functor of \cref{cons:Tbifib}. We write simply $T$ when the indices are clear from the context.
\end{notation}

We will construct the free $\cL$-bifibration as a localization of $\cX_{T}$ (cf. \cref{cons:Tbifib}) at a certain collection of 1 and 2-morphisms which we introduce below.


\begin{construction}\label{cons:fibrewisejunk}
	Fix $0\leq k\leq m$. The source and target of
	\[
		T_{2k,2m}\colon\cX^{(2k)}\longrightarrow\cX^{(2m)}
	\]
	are $(i,j)$-fibrations, while the source and target of $T_{2k+1,2m+1}$ are partial $(1-i,j)$-fibrations relative to $\cL$. We describe the even case for $(i,j)=(0,1)$; the remaining cases are dual.

	Given a cocartesian morphism $e\colon x^{2k}\to y^{2k}$, factor $T_{2k,2m}(e)$ as a cocartesian morphism $\widetilde e$ followed by a fibrewise morphism $w_e$:
	\[
	\begin{tikzcd}
		& y^{2m} \arrow[dr,"w_e"] & \\
		T_{2k,2m}(x^{2k}) \arrow[ur,"\widetilde e"]
			\arrow[rr,"T_{2k,2m}(e)"']
		&& T_{2k,2m}(y^{2k}).
	\end{tikzcd}
	\]
	Here $\widetilde e$ is a cocartesian lift of $p^{(2k)}(e)$ and $w_e$ lies over the identity of its target. A totally analogous construction of $w_e$ can be performed in the odd case for cartesian $1$-morphisms living over $\cL$.

	A similar construction applies one categorical level higher. For $0\leq\ell\leq2m$, given a cartesian $2$-morphism $\varphi\colon f\to g$ between parallel $1$-morphisms $f,g\colon a\to b$ in $\cX^{(\ell)}$, factor $T_{\ell,2m}(\varphi)$ as
	\[
		T_{\ell,2m}(f)
		\xrightarrow{\alpha_\varphi} f^{2m}
		\xrightarrow{\widetilde\varphi} T_{\ell,2m}(g),
	\]
	where $\widetilde\varphi$ is cartesian and $\alpha_\varphi$ lies in
	\[
		\cX^{(2m)}(T_{\ell,2m}(a),T_{\ell,2m}(b))_{p(f)}.
	\]
	Choosing a cocartesian lift $T_{\ell,2m}(a)\to\widehat b$ over $p(f)$ gives an equivalence
	\[
		\cX^{(2m)}_{p(b)}(\widehat b,T_{\ell,2m}(b))
		\xrightarrow{\simeq}
		\cX^{(2m)}(T_{\ell,2m}(a),T_{\ell,2m}(b))_{p(f)},
	\]
	under which $\alpha_\varphi$ corresponds to a $2$-morphism in the fibre over $p(b)$.

\end{construction}

\begin{definition}\label{def:comparison}
	We call the fibrewise $r$-morphisms of \cref{cons:fibrewisejunk} \emph{comparison $r$-morphisms}.
\end{definition}

\begin{definition}\label{def:cwclean}
	For each $k\geq0$, consider the fibrewise comparison morphisms of \cref{cons:fibrewisejunk} associated to the transition functors
	\[
		T_{\ell,k}\colon\cX^{(\ell)}\longrightarrow\cX^{(k)},
		\qquad 0\leq\ell<k.
	\]
	We define a cospan $\cV_k^1 \xrightarrow{}\cX^{(k)} \xleftarrow{} \cV_k^2$, where (cf. \cref{def:typeofsubcat}):
	\begin{itemize}
		\item $\cV_k^1$ is the smallest wide locally full sub-$(\infty,2)$-category containing the comparison $1$-morphisms for $\ell\equiv k\pmod 2$;
		\item for even $k$, $\cV_k^2$ is the smallest wide and locally wide sub-$(\infty,2)$-category containing the comparison $2$-morphisms for every $\ell<k$; for odd $k$, it is the wide and locally wide sub-$(\infty,2)$-category containing only invertible $2$-morphisms.
	\end{itemize}
\end{definition}

For the rest of this section we will assume without loss of generality that $(i,j)=(0,1)$. The remaining variances follow by formally dual arguments.

\begin{definition}\label{def:cwfinal}
	 For $r=1,2$, set $\cW_0^r=\cV_0^r$. Assuming that $\cW_k^r$ has been defined for $k\geq0$, define:
	\begin{itemize}
		\item $\cW_{k+1}^{1}$ to be the smallest wide locally full sub-$(\infty,2)$-category of $\cX^{(k+1)}$ containing $\cV^{1}_{k+1}$ and $T_{k,k+1}(\cW_{k}^{1})$ and whose fibrewise $1$-morphisms are stable under cocartesian transport when $k+1$ is even, and under cartesian transport along morphisms in $\cL$ when $k+1$ is odd.
		\item $\cW_{k+1}^{2}$ to be the smallest wide and locally wide sub-$(\infty,2)$-category of $\cX^{(k+1)}$ containing $\cV^{2}_{k+1}$ and $T_{k,k+1}(\cW_{k}^{2})$ and whose fibrewise $2$-morphisms are stable under cocartesian transport when $k+1$ is even, and under cartesian transport along morphisms in $\cL$ when $k+1$ is odd.
	\end{itemize}
	For $r=1,2$, taking the corresponding subcategories of $\cX_T$ generated by the images of the $\cW_k^r$ for $k\geq0$, we obtain a cospan $\cW^1 \xrightarrow{} \cX_T \xleftarrow{} \cW^2$, where the left leg is a wide locally full inclusion and the right leg is a wide and locally wide inclusion.
\end{definition}

\begin{definition}\label{def:diamond}
	Given $p\colon \cX_T \to \cC$ as before, let $p_{\diamond} \colon \cX_{\diamond} \to \cC$ denote the functor obtained by inverting every $1$-morphism belonging to $\cW^1$ and every $2$-morphism belonging to $\cW^{2}$ (see \cref{def:cwfinal}). 
\end{definition}

\begin{notation}\label{not:diamondnot}
	We fix some notation for the remainder of the section:
	\begin{itemize}
		\item Let $\cX^{(k)}_{\diamond}$ denote the localization of $\cX^{(k)}$ at every $1$-morphism belonging to $\cW_{k}^{1}$ and every $2$-morphism belonging to $\cW_{k}^{2}$.
		\item For $c\in\cC$, $k\geq0$, and $r=1,2$, write $\cW_{k,c}^{r}=\cW_k^r\times_{\cC}\{c\}$. Taking the corresponding sub-$(\infty,2)$-categories generated by their images for all $k\geq0$, we obtain a cospan
	\[
		\cW_c^1 \xrightarrow{} \cX_T\times_{\cC}\{c\} \xleftarrow{} \cW_c^2,
	\]
	   where the left leg is a wide and locally full sub-$(\infty,2)$-category and the right leg is a wide locally wide sub-$(\infty,2)$-category.
		\item We denote by $\cX^{(k)}_{\diamond,c}$ the localization of $\cX^{(k)}_c=\cX^{(k)}\times_{\cC}\{c\}$ at the $1$-morphisms of $\cW_{k,c}^{1}$ and the $2$-morphisms of $\cW_{k,c}^{2}$.
	\end{itemize}
	
\end{notation}

\begin{proposition}\label{prop:new01}
	The functor $p_{\diamond} \colon \cX_{\diamond} \to \cC$ is a $(0,1)$-fibration and every morphism $\cX^{(2k)} \to \cX_{\diamond}$ is a morphism of $(0,1)$-fibrations for $k\geq 0$.
\end{proposition}
\begin{proof}
	 Since the fibrewise morphisms in $\cW_{2k}^{r}$ are stable under the cocartesian transport, it follows from \cite[Theorem 4.7.1]{AHM26} that the corresponding functors $\cX^{(2k)}_\diamond \to \cC$ are $(0,1)$-fibrations for $k\geq 0$ (cf. \cref{not:diamondnot}). By construction we have that $T_{k,k+1}(\cW_{k}^{r})\subset \cW_{k+1}^{r}$ which shows that we have functors $\cX_{\diamond}^{(k)} \to \cX_{\diamond}^{(k+1)}$ over $\cC$. We conclude that 
	\[
		\cX_{\diamond}\xrightarrow{\simeq}\colim_{k \in \mathbb{N}}\cX_{\diamond}^{(k)} \xrightarrow{\simeq}\colim_{k \in \mathbb{N}}\cX_{\diamond}^{(2k)},
	\]
	where the final step uses cofinality of the even integers.
	To finish the proof we must show that each $\cX_{\diamond}^{(2k)} \to \cX_{\diamond}^{(2k+2)}$ defines a morphism of $(0,1)$-fibrations. This is clear since $\cW^{r}_{2k+2}$ contains the relevant comparison morphisms of \cref{cons:fibrewisejunk}.
\end{proof}

\begin{remark}\label{prop:fibresasexpected}
	 Let $\cX^{(k)}_{\diamond,c}$ be as in \cref{not:diamondnot}. Since filtered colimits commute with finite limits of $(\infty,2)$-categories, we have
	\[
		\cX_{\diamond}\times_{\cC}\{c\} \simeq \colim_{k\in\mathbb N}\bigl(\cX_{\diamond}^{(2k)}\times_{\cC}\{c\}\bigr)\simeq \colim_{k\in\mathbb N}\cX^{(2k)}_{\diamond,c} \simeq \bigl(\cX_{T}\times_{\cC}\{c\}\bigr)_{\diamond}.
	\]

	The second equivalence follows from \cite[Theorem 4.7.1]{AHM26}, and the last from compatibility of localizations with filtered colimits and cofinality of the even integers. As above, the final $(\infty,2)$-category is obtained from the fibre $\cX_{T}\times_{\cC}\{c\}$, by inverting the $1$-morphisms of $\cW_c^1$ and the $2$-morphisms of $\cW_c^2$.
\end{remark}

By \cref{prop:charactbifib}, showing that $\cX_\diamond$ is an $\cL$-bifibration amounts to proving that cocartesian transport along every morphism in $\cL$ admits a right adjoint. To construct these adjoints, we will exploit the description of the fibres of $\cX_{\diamond}$ from \cref{prop:fibresasexpected}.

Let us fix $f\colon a\to b$ in $\cL$. We now construct a right adjoint to the cocartesian transport functor
\begin{equation}\label{eq:transp}
	f_!\colon\cX_{\diamond}\times_{\cC}\{a\}\longrightarrow\cX_{\diamond}\times_{\cC}\{b\}
\end{equation}

This will be achieved by a level-wise construction.

\begin{construction}\label{cons:pushL}
	For $k\geq0$, set
	\[
		\tau_k=
		\begin{cases}
			k & k\text{ even},\\
			k+1 & k\text{ odd}.
		\end{cases}
	\]
	Define $f_L^{(k)}\colon\cX_a^{(k)}\to\cX_b^{(\tau_k)}$ to be cocartesian transport when $k$ is even, and set
	\[
		f_L^{(k)}=f_L^{(k+1)}\circ T_{k,k+1}
	\]
	when $k$ is odd.
\end{construction}

\begin{definition}\label{def:cartpush}
	For $k\geq0$, set
	\[
		\mu_k=
		\begin{cases}
			k & k\text{ odd},\\
			k+1 & k\text{ even}.
		\end{cases}
	\]
	Define $f_R^{(k)}\colon\cX_b^{(k)}\to\cX_a^{(\mu_k)}$ to be cartesian transport when $k$ is odd, and set
	\[
		f_R^{(k)}=f_R^{(k+1)}\circ T_{k,k+1}
	\]
	when $k$ is even.
\end{definition}

\begin{remark}\label{rem:laxlycommuting}
	For $k\leq r$, the functors $f_L^{(k)}$ fit into laxly commutative squares
	\[
	\begin{tikzcd}
		\cX_a^{(k)} \arrow[r] \arrow[d,"f_L^{(k)}"']
		& \cX_a^{(r)} \arrow[d,"f_L^{(r)}"] \\
		\cX_b^{(\tau_k)} \arrow[r]
		& \cX_b^{(\tau_r)},
		\arrow[between={0.3}{0.8},Leftarrow,from=2-1,to=1-2]
	\end{tikzcd}
	\]
	while the $f_R^{(k)}$ fit into the dual oplaxly commutative squares
	\[
	\begin{tikzcd}
		\cX_b^{(k)} \arrow[r] \arrow[d,"f_R^{(k)}"']
		& \cX_b^{(r)} \arrow[d,"f_R^{(r)}"] \\
		\cX_a^{(\mu_k)} \arrow[r]
		& \cX_a^{(\mu_r)}.
		\arrow[between={0.2}{0.8},Leftarrow,from=1-2,to=2-1]
	\end{tikzcd}
	\]
	It is enough to consider $r=k+1$. One parity gives a strictly commutative square by definition; in the other, the desired transformation is the lax or oplax transformation on fibres corresponding to $T_{k,k+2}$, by \cite[Theorem 3.5.4]{AbellanGagnaea2025StraighteningLaxTransformations}.
\end{remark}

\begin{proposition}\label{prop:identifiedf!}
	After passage to localizations, the functors $f_L^{(k)}$ induce the cocartesian transport functor $f_! \colon \cX_{\diamond}\times_{\cC}\{a\}\longrightarrow\cX_{\diamond}\times_{\cC}\{b\}.$
\end{proposition}
\begin{proof}
	We make use of \cref{not:diamondnot} throughout the proof.
	 By construction (see \cref{def:cwfinal}), the functors $f_L^{(k)}$ descend to the stagewise localizations. We claim that the induced functor
	\[
		\colim_{k \in \mathbb{N}}\cX_{\diamond,a}^{(k)} \longrightarrow\colim_{k \in \mathbb{N}}\cX_{\diamond,b}^{(\tau_k)}
	\]
	can be identified with $f_!$. By cofinality, we may restrict to the even integers. By \cref{prop:new01}, we have a commutative diagram
	\[
		\begin{tikzcd}
			\cX^{(2k)}_{\diamond,a} \arrow[r] \arrow[d] & \cX^{(2k)}_{\diamond,b} \arrow[d] \\
			\cX_{\diamond}\times_{\cC}\{a\} \arrow[r,"f_!"] & \cX_{\diamond}\times_{\cC}\{b\}.
		\end{tikzcd}
	\]
	By \cref{prop:fibresasexpected}, the vertical functors become equivalences after passage to filtered colimits, proving the claim.
\end{proof}

\begin{proposition}\label{prop:f^*descends}
	After passage to the stagewise localizations, the functors $f_R^{(k)}$ induce a functor $f^*\colon\cX_{\diamond}\times_{\cC}\{b\}\longrightarrow\cX_{\diamond}\times_{\cC}\{a\}.$
\end{proposition}
\begin{proof}
    By the construction of $\cW_{k}^{r}$ the cartesian transport functor descends to a map (see \cref{not:diamondnot})
	\[
		\cX_{\diamond,b}^{(2k+1)} \to \cX_{\diamond,a}^{(2k+1)}.
	\]
	Moreover, since the components of the comparison transformation in the oplax-commutative square from \cref{rem:laxlycommuting} are inverted by the localization we obtain a morphism after passage to  filtered colimits
	\[
		\colim_{k \in \mathbb{N}}\cX_{\diamond,b}^{(2k+1)} \to \colim_{k \in \mathbb{N}}\cX_{\diamond,a}^{(2k+1)}.
	\]
	Using cofinality of the odd integers, the claim now follows from \cref{prop:fibresasexpected}.
\end{proof}

We next construct the unit and counit. We give the construction of the unit in even degrees; the remaining cases are dual.

\begin{construction}\label{cons:eta}
	Let $H\colon\cX_a^{(2k)}\times[1]\to\cX^{(2k)}$ be the natural transformation whose components are cocartesian lifts of $f\colon a\to b$. Apply $T_{2k,2k+1}$ and extend the resulting transformation to
	\[
		\Phi\colon\cX_a^{(2k)}\times[2]\longrightarrow\cX^{(2k+1)}
	\]
	so that its restriction to $\{0<2\}$ is $T_{2k,2k+1}\circ H$ and its restriction to $\{1<2\}$ consists of cartesian morphisms; the extension exists by \cref{prop:fibfunct}. Restriction to $\{0<1\}$ factors through the fibre over $a$ and defines
	\[
		\eta^{(2k)}\colon
		T_{2k,2k+1}
		\longrightarrow
		f_R^{(2k)}f_L^{(2k)}.
	\]
	The odd-degree units and the counits
	\[
		\epsilon^{(k)}\colon
		f_L^{(\mu_k)}f_R^{(k)}\longrightarrow T_{k,\tau_{\mu_k}}
	\]
	are defined dually.
\end{construction}

\begin{proposition}\label{prop:etadescends}
	After passage to localizations, the transformations $\eta^{(k)}$ and $\epsilon^{(k)}$ descend to natural transformations
	\[
		\eta\colon\id\longrightarrow f^*f_!,
		\qquad
		\epsilon\colon f_!f^*\longrightarrow\id.
	\]
\end{proposition}
\begin{proof}
	By \cref{prop:fibresasexpected}, it is enough to verify compatibility with the transition functors. We treat $\eta^{(2k)}$; the other cases are dual. The lax and oplax squares of \cref{rem:laxlycommuting} give a diagram
	\[\begin{tikzcd}
	& {T_{2k+1,2k+3}\circ T_{2k,2k+1}} & \\
	{T_{2k+1,2k+3}\circ f_{R}^{(2k)}\circ f_L^{(2k)}} && {f_R^{(2k+2)}f_{L}^{(2k+2)}\circ T_{2k,2k+2}} \\
	& {f_R^{(2k+2)}\circ T_{2k,2k+2}\circ f_L^{(2k)}}
	\arrow[from=1-2, to=2-1]
	\arrow[from=1-2, to=2-3]
	\arrow[from=2-1, to=3-2]
	\arrow[from=2-3, to=3-2]
\end{tikzcd}\]
	Both composites are obtained from the same transformation
	\[
		T_{2k,2k+3}\circ H\colon
		\cX_a^{(2k)}\times[1]\longrightarrow\cX^{(2k+3)}
	\]
	by factoring it along a cartesian morphism in
	$\Fun(\cX_a^{(2k)},\cX^{(2k+3)})$. Such a cartesian factorization is unique up to a contractible space of choices, so the two composites agree. 
\end{proof}

\begin{proposition}\label{prop:isbifib}
	The functor $p_\diamond \colon \cX_\diamond \to \cC$ is an $\cL$-bifibration and each morphism $\cX^{(2k+1)} \to \cX_{\diamond}$ is a morphism of partial $(1,1)$-fibrations relative to $\cL$ for $k\geq 0$.
\end{proposition}
\begin{proof}
	 By \cref{prop:identifiedf!,prop:f^*descends,prop:etadescends}, we have functors
	\[
		f_!\colon \cX_{\diamond}\times_{\cC} \{a\} \llra
		\cX_{\diamond}\times_{\cC} \{b\}\colon f^*
	\]
	and transformations $\eta\colon\id\to f^*f_!$ and $\epsilon\colon f_!f^*\to\id$. It remains to verify the triangular identities. By \cref{prop:fibresasexpected}, it is enough to do so stagewise. For instance, the composite
	\[
		f_L^{(2k+1)}T
		\longrightarrow
		f_L^{(2k+1)}f_R^{(2k)}f_L^{(2k)}
		\longrightarrow
		Tf_L^{(2k)}
	\]
	becomes, after using $f_L^{(2k+1)}=f_L^{(2k+2)}T_{2k+1,2k+2}$, the comparison transformation
	\[
		f_L^{(2k+2)}T_{2k,2k+2}
		\longrightarrow
		T_{2k,2k+2}f_L^{(2k)}.
	\]
	Its components are comparison morphisms and hence become invertible after localization. The other triangular identity is dual. Thus $f_!\dashv f^*$, and \cref{prop:charactbifib} shows that $\cX_\diamond\to\cC$ is an $\cL$-bifibration. By construction the maps $\cX^{(2k+1)} \to \cX_{\diamond}$ preserve cartesian $1$-morphisms over $\cL$. Given a cartesian 2-morphism in $\cX^{(2k+1)}$ we obtain a comparison 2-morphism in $\cX^{(2k+2)}$ (cf. \cref{def:cwclean}) which becomes invertible in $\cX_{\diamond}$. Hence our functor is a morphism of partial $(1,1)$-fibrations relative to $\cL$.
\end{proof}

\medskip
\begin{proof}[Proof of \cref{thm:bifib}]
	We consider $(i,j)=(0,1)$ without loss of generality. Consider the canonical morphism over $\cC$, $\eta_{\cX}\colon \cX \to \cX_{\diamond}$. To prove the statement we need to show that restriction along $\eta_{\cX}$ gives an equivalence of $(\infty,2)$-categories
	\[
		\eta_{\cX}^* \colon \Fun_{/\cC}^{\cL\mhyphen\bifib}(\cX_{\diamond},\cY) \xrightarrow{\simeq}\Fun^{(0,1)}_{/\cC}(\cX,\cY),
	\]
	for every $\cL$-bifibration $\cY \to \cC$. Recall that $\Fun_{/\cC}^{\cL\mhyphen\bifib}(\cX_{\diamond},\cY)$ sits fully faithfully inside the $(\infty,2)$-category $\Fun_{/\cC}(\cX_{\diamond},\cY)$ of all functors over $\cC$. By the universal property of localization, we obtain a fully faithful functor $\omega_{\cX} \colon \Fun_{/\cC}^{\cL\mhyphen\bifib}(\cX_{\diamond},\cY) \hookrightarrow \lim_{\mathbb{N}^\op}\Fun_{/\cC}(\cX^{(k)},\cY) \simeq \Fun_{/\cC}(\cX_T,\cY)$. We claim that its essential image consists of those functors $\cX_T\to\cY$ whose restriction to $\cX^{(k)}$ is a morphism of $(0,1)$-fibrations for even $k$ and of partial $(1,1)$-fibrations relative to $\cL$ for odd $k$.

	One direction follows from \cref{prop:new01,prop:isbifib}. Conversely, let $\varphi\colon\cX_T\to\cY$ satisfy this property. Since its restrictions preserve the relevant lifts, $\varphi$ inverts the comparison $r$-morphisms generating $\cV_k^r$. Compatibility with transport then shows that it also inverts the $r$-morphisms of $\cW_k^r$, for $r=1,2$. Thus $\varphi$ descends to $\cX_\diamond$. One checks easily that the descended functor is a morphism of $\cL$-bifibrations.

	For $k\geq0$, write $\Fun^{(0,1)}_{/\cC}(\cX^{(2k)},\cY)^{\star}$ and $\Fun^{(1,1)}_{/(\cC,\cL)}(\cX^{(2k+1)},\cY)^{\star}$ for the full sub-$(\infty,2)$-categories on those functors whose restrictions to every earlier stage preserve the relevant (co)cartesian $r$-morphisms, according to its parity. By the previous discussion, the limit of these $(\infty,2)$-categories over $\mathbb N^\op$ is identified with $\Fun_{/\cC}^{\cL\mhyphen\bifib}(\cX_{\diamond},\cY)$. By the universal property of the free (partial) fibration from \cref{thm:pfree}, restriction gives equivalences, for $k\geq1$,
	\[
		\Fun^{(0,1)}_{/\cC}(\cX^{(2k)},\cY)^{\star} \xrightarrow{\simeq}\Fun^{(1,1)}_{/(\cC,\cL)}(\cX^{(2k-1)},\cY)^{\star},
	\]
	and similarly from stage $2k+1$ to stage $2k$ for $k\geq0$. Piecing this together shows that $\Fun_{/\cC}^{\cL\mhyphen\bifib}(\cX_{\diamond},\cY) \simeq \Fun^{(0,1)}_{/\cC}(\cX,\cY)$, where this equivalence is induced by $\eta_\cX^*$ and thus natural in $\cY$.
\end{proof}

\section{Freely adding adjoints to an $(\infty,2)$-category}

\begin{definition}\label{def:leftlocalization1}
	Let $\cL \subset \cC$ be a wide locally full sub-$(\infty,2)$-category inclusion and consider the functor
	\[
		\Fun(\cC,-)_{\cL} \colon \Cat_{(\infty,2)} \to \Cat_{(\infty,2)},
	\]
	sending each $(\infty,2)$-category $\cX$ to the locally full sub-$(\infty,2)$-category of $\Fun(\cC,\cX)$ consisting of:
	\begin{itemize}
		\item Those functors $F \colon \cC \to \cX$ sending each 1-morphism in $\cL$ to a left adjoint in $\cX$.
		\item Those natural transformations $F \to G$ such that for every $f \colon a \to b$ in $\cL$ the naturality square
		\[
			\begin{tikzcd}
				F(a) \arrow[d] \arrow[r,"F(f)"] & F(b) \arrow[d] \\
				G(a) \arrow[r,"G(f)"] & G(b)
			\end{tikzcd}
		\]
		is horizontally adjointable (\cref{def:adjointable}).
	\end{itemize}
\end{definition}

\begin{proposition}\label{prop:softlocalization}
	Let $\cL\subset \cC$ as in \cref{def:leftlocalization1}. Then the functor $\Fun(\cC,-)_{\cL}$ is corepresentable by an $(\infty,2)$-category $\cC\soft{\cL}$.
\end{proposition}
\begin{proof}
	To show our statement it suffices to prove $\Fun(\cC,-)_{\cL}$ is a right adjoint of $(\infty,3)$-categories, since the value of the left adjoint at the terminal category $[0]$ will then corepresent this functor. First, we will show that $\Fun(\cC,-)_{\cL}$ admits a left adjoint when viewed as a functor on underlying $(\infty,1)$-categories. Since both the domain and codomain are given by presentable $(\infty,1)$-categories, invoking Lurie's adjoint functor theorem, \cite[Corollary 5.5.2.9]{HTT}, it will be enough to verify that our functor preserves limits and $\kappa$-filtered colimits for some regular cardinal $\kappa$. The claim regarding filtered colimits follows easily by picking some regular cardinal $\kappa$ such that $\cC$ and $\cL$ are both $\kappa$-compact in (underlying $(\infty,1)$-category of) $\Cat_{(\infty,2)}$. The claim regarding limits follows by direct inspection. To finish the proof one observes that $\Fun(\cC,-)_{\cL}$ preserves cotensors by any $(\infty,2)$-category due to \cite[Corollary 4.2.12]{AbellanGagnaea2025StraighteningLaxTransformations} which allows us to upgrade the adjunction to an adjunction of $(\infty,3)$-categories.
\end{proof}

\begin{remark}\label{rem:softstrongloc}
	Suppose that $\cC$ is an $(\infty,1)$-category. By design, we have that localization of $\cC\soft{\cL}$ at all 2-morphisms yields $\cC[\cL^{-1}]$. 
\end{remark}

\begin{lemma}\label{lem:softesssurj}
	The canonical functor $\iota \colon \cC \to \cC\soft{\cL}$ is essentially surjective.
\end{lemma}
\begin{proof}
	Let $\beta \colon \cX \subset \cC\soft{\cL}$ denote the essential image of $\iota$. Then the induced functor $j \colon \cC \to \cX$ sends each morphism in $\cL$ to a left adjoint and thus one obtains a factorization $ \alpha \colon  \cC\soft{\cL} \to \cX$ such that $\alpha \circ \iota=j$. We claim that $\beta \circ \alpha=\id$, but this can be checked after precomposition with $\iota$ and thus we get $\beta \circ \alpha \circ \iota= \beta \circ j=\iota$ which shows the claim.  We conclude that $\beta$ is essentially surjective as desired.
\end{proof}

\begin{proposition}\label{prop:1dimnfibres}
	Let $p \colon \cX \to \cC$ be an $(i,j)$-fibration whose fibres are $(\infty,1)$-categories. Then $\BFree^{\epsilon}_{\cL}(p)\colon \BFree^{\epsilon}_{\cL}(\cX) \to \cC$ is also fibred in $(\infty,1)$-categories.
\end{proposition}
\begin{proof}
   Recall that $\tau_1 \colon \Cat_{(\infty,2)} \to \Cat_{(\infty,1)}$ (cf. \cref{not:adjointstring}) is a left adjoint of $(\infty,2)$-categories and so in particular, it preserves adjunctions.
	
   Write $\tau^{\epsilon}_1 \BFree^{\epsilon}_{\cL}(\cX) \to \cC$, for the $\epsilon$-fibration obtained by fibre-wise applying the functor $\tau_1$. Since $\tau_1$ is a functor of $(\infty,2)$-categories, it follows that $\tau^{\epsilon}_1 \BFree^{\epsilon}_{\cL}(\cX) \to \cC$ is a $\cL$-bifibration. Let us now fix some $\cL$-bifibration $\cY \to \cC$. We have  natural equivalences of $\infty$-groupoids
	\[
		\Fun_{/\cC}^{\cL\mhyphen\bifib}(\tau_1^{\epsilon}\BFree^{\epsilon}_{\cL}(\cX),\cY)^{\simeq}\simeq \Fun_{/\cC}^{\cL\mhyphen\bifib}(\tau_1^{\epsilon}\BFree^{\epsilon}_{\cL}(\cX),\cY^{\leq 1,\epsilon})^{\simeq} 
	\]
	\[
		\xrightarrow{\simeq}\Fun_{/\cC}^{\cL\mhyphen\bifib}(\BFree^{\epsilon}_{\cL}(\cX),\cY^{\leq 1,\epsilon})^{\simeq}\simeq \Fun_{/\cC}^{\epsilon}(\cX,\cY^{\leq 1,\epsilon})^{\simeq} \simeq \Fun_{/\cC}^{\epsilon}(\cX,\cY)^{\simeq}
	\]
	where $\cY^{\leq 1,\epsilon} \to \cC$ is the $\cL$-bifibration obtained by applying the functor $(-)^{\leq 1}$ in a fibre-wise manner. Note that by definition we also have a natural equivalence
	\[
		\Fun_{/\cC}^{\epsilon}(\cX,\cY)^{\simeq} \simeq \Fun_{/\cC}^{\cL\mhyphen\bifib}(\BFree^{\epsilon}_\cL(\cX),\cY)^{\simeq}.
	\]
     The Yoneda lemma now yields an equivalence of $\cL$-bifibrations, $\BFree^{\epsilon}_\cL(\cX)\simeq \tau^{\epsilon}_1 \BFree^{\epsilon}_{\cL}(\cX)$ thus establishing the claim.
\end{proof}

We now arrive at the main theorem of this section.

\begin{theorem}\label{thm:addingadjoints}
	The $(\infty,2)$-category $\cC\soft{\cL}$ is equivalent to the full sub-$(\infty,2)$-category of $\Fib^{(1,0)\mhyphen\bifib}_{/(\cC,\cL)}$ spanned by the objects $\BFree^{(1,0)}_{\cL}(\cC_{\upslash c}) \to \cC$, for $c\in \cC$.
\end{theorem}
\begin{proof}
	Let us consider the Yoneda embedding
	\[
		\cC\soft{\cL} \to \Fun(\cC\soft{\cL}^\op,\Cat_{(\infty,1)}),
	\]
	together with the chain of equivalences
	 \[
		\Fun(\cC\soft{\cL}^\op,\Cat_{(\infty,1)}) \simeq \Fun(\cC\soft{\cL},\Cat_{(\infty,1)}^\op)^{\op} \simeq (\Fun(\cC,\Cat_{(\infty,1)}^{\op})_{\cL})^\op.
	 \]
	 From these we conclude that $\Fun(\cC\soft{\cL}^\op,\Cat_{(\infty,1)}) \hookrightarrow \Fib^{(1,0)\mhyphen\bifib}_{/(\cC,\cL)}$ is fully-faithful with essential image given by those $\cL$-bifibrations with $(\infty,1)$-categorical fibres.

	 Let $\iota \colon \cC \to \cC\soft{\cL}$ denote the canonical functor. Then the image of $\iota(c) \in \cC\soft{\cL}$ in $ \Fib^{(1,0)\mhyphen\bifib}_{/(\cC,\cL)}$ is identified with $\BFree^{(1,0)}_{\cL}(\cC_{\upslash c})$: indeed, the Yoneda lemma together with the universal property of the free bifibration shows that mapping out of both objects represents the functor which evaluates at $c$, so the two must agree. By \cref{lem:softesssurj}, $\cC \to \cC\soft{\cL}$ is essentially surjective, which completes the proof.
\end{proof}

\begin{definition}\label{def:Bfreeloc}
	Let $\cL \subset \cC$. Given $\BFree^{(1,0)}_{\cL}(\cC_{\upslash c}) \to \cC$, we define $\BFree^{(1,0)\mhyphen\loc}_{\cL}(\cC_{\upslash c}) \to \cC$ by localizing those fibre-wise morphisms (and their cartesian transports) that represent the units and counits of the adjunctions associated to the morphisms in $\cL$.
\end{definition}

\begin{corollary}\label{cor:dwyerkan}
	 Then the $(\infty,2)$-category $\cC[\cL^{-1}]$ is equivalent to the full sub-$(\infty,2)$-category of $\Fib^{(1,0)\mhyphen\bifib}_{/(\cC,\cL)}$ spanned by the objects $\BFree^{(1,0)\mhyphen\loc}_{\cL}(\cC_{\upslash c}) \to \cC$, for $c\in \cC$ (see \cref{def:Bfreeloc}).
\end{corollary}
\begin{proof}
	We define $\Fib^{(1,0)\mhyphen\mathrm{loc}}_{/(\cC,\cL)}$ to be the full sub-$(\infty,2)$-category of $\Fib^{(1,0)\mhyphen\bifib}_{/(\cC,\cL)}$ spanned by those $\cL$-bifibrations with the following property:
	\begin{itemize}
		\item A 1-morphism over $\cL$ is cartesian if and only if it is cocartesian.
	\end{itemize}
	Invoking \cref{thm:addingadjoints}, we obtain a fully faithful functor of $(\infty,2)$-categories 
	\[
		\Fun(\cC[\cL^{-1}]^{\op},\Cat_{(\infty,1)})\hookrightarrow\Fib^{(1,0)\mhyphen\mathrm{loc}}_{/(\cC,\cL)} ,
	\]
	 with essential image given by those fibrations with $(\infty,1)$-categorical fibres. Invoking the Yoneda embedding, we conclude that $\cC[\cL^{-1}] \hookrightarrow{} \Fib^{(1,0)\mhyphen\mathrm{loc}}_{/(\cC,\cL)}$ lands in the full sub-$(\infty,2)$-category generated by the various $\BFree^{(1,0)\mhyphen\loc}_{\cL}(\cC_{\upslash c})$. The result follows since $\cC \to \cC[\cL^{-1}]$ is essentially surjective.
\end{proof}

\subsection{A formula for the mapping $(\infty,1)$-category}\label{subsec:mappingformula}
Recall that by design, for $x,y \in \cC\soft{\cL}$ there is an equivalence
\[
	\cC\soft{\cL}(x,y) \simeq \BFree^{(1,0)}_{\cL}(\cC_{\upslash y})\times_{\cC}\{x\}.
\]
Since the free $\cL$-bifibration is given by an explicit construction, this equivalence yields a concrete formula for the mapping $(\infty,1)$-categories of $\cC\soft{\cL}$. In  this section we unravel this formula when $\cC$ and $\cL$ are $(\infty,1)$-categories. The same method applies in general, but this case displays the essential features of the formula more transparently and is sufficient for the applications to come.

\begin{definition}\label{def:zigzag}
	For $n \geq 0$, let $Z_{n}$ denote the poset with objects $\{0,1,\dots,n\}$ whose only non-identity morphisms are, for each $1 \leq i \leq n$, a single morphism between $i-1$ and $i$: namely $i-1 \to i$ if $i$ is odd, and $i \to i-1$ if $i$ is even. In other words, $Z_n$ is an alternating zig-zag of length $n$, starting with $0\to1$ when $n\geq1$.
\end{definition}

\begin{remark}\label{rem:zig1}
	Let $\cL, \cC$ be $(\infty,1)$-categories and recall the dual of \cref{cons:Tbifib}. We observe that for every $n\geq 1$ we have a pullback diagram
	\[
		\begin{tikzcd}
			\cX^{(n)} \arrow[d] \arrow[r] & \Fun_{\cL}(Z_n,\cC) \arrow[d,"\ev_0"] \\
			\cX \arrow[r,"p"] & \cC.
		\end{tikzcd}
	\]
	where $\Fun_{\cL}(Z_n,\cC)$ denotes the full sub-$(\infty,1)$-category on those functors sending the right-pointing morphisms in $Z_n$ to $\cL$. The functor $p^{(n)} \colon \cX^{(n)} \to \cC$ is induced by $\ev_n \colon \Fun(Z_n,\cC) \to \cC$.
\end{remark}

\begin{remark}\label{rem:zig2}
	For every $n \geq 0$ there is a functor $t_n \colon Z_{n+1} \to Z_n$, collapsing the morphism between $n$ and $n+1$ down to the object $n$. Restricting along $t_n$ then gives a functor $t_n^* \colon \Fun_{\cL}(Z_n,\cC) \to \Fun_{\cL}(Z_{n+1},\cC)$, which is compatible with the projections $\ev_0$ to $\cC$. It follows that the functor $\cX^{(n)} \to \cX^{(n+1)}$ from \cref{cons:Tbifib} is simply the pullback of $t_n^*$ along $p$.
\end{remark}

\begin{definition}
	Let $Z_{n}^{\star}=\{-1 \xleftarrow{}0\} \coprod_{\{0\}} Z_n$ and denote by $\Fun_{\cL}(Z_n^{\star},\cC)$ the pullback of the cospan $\Fun(Z_n^{\star},\cC) \xrightarrow{}{\Fun(Z_n,\cC)}\xleftarrow{} \Fun_{\cL}(Z_n,\cC)$.
\end{definition}

Applying the preceding description to the source projection $\cC_{/y}\to\cC$, we obtain a pullback diagram
 \[
	\begin{tikzcd}
		{\cC_{/y}^{(n)}} \arrow[r] \arrow[d] & \Fun_{\cL}(Z_n^{\star},\cC) \arrow[d,"\ev_{-1}"] \\
		{[0]} \arrow[r,"y"] & \cC
	\end{tikzcd}
 \]
 \begin{definition}
	For $x\in\cC$, let $\cC^{(n)}(x,y)$ denote the fibre of $\cC_{/y}^{(n)}\to\cC$ over $x$. Its objects are alternating zig-zags of length $n+1$ from $x$ to $y$ (note that we are reading the zig-zag from right to left) whose right-pointing morphisms belong to $\cL$, and its morphisms are maps of such zig-zags. A morphism is depicted schematically below for even $n\geq4$, with both endpoints fixed:
 \[\begin{tikzcd}
	{y} & {y_0} & {y_1} & \cdots & {y_{n-1}} & {x} \\
	& {z_0} & {z_1} & \cdots & {z_{n-1}}
	\arrow[from=1-2, to=1-1]
	\arrow[from=1-2, to=1-3]
	\arrow[from=1-2, to=2-2]
	\arrow[from=1-3, to=2-3]
	\arrow[from=1-4, to=1-3]
	\arrow[from=1-4, to=2-4]
	\arrow[from=1-4, to=1-5]
	\arrow[from=1-6, to=1-5]
	\arrow[from=1-5, to=2-5]
	\arrow[from=2-2, to=1-1]
	\arrow[from=2-2, to=2-3]
	\arrow[from=2-4, to=2-3]
	\arrow[from=2-4, to=2-5]
	\arrow[from=1-6, to=2-5]
\end{tikzcd}\]
 \end{definition}

 \begin{notation}\label{not:cinfty}
	We denote $\cC^{\infty}_{/y}= \colim_{\mathbb{N}}\cC^{(n)}_{/y}$ and  $\cC^{\infty}(x,y)= \colim_{\mathbb{N}}\cC^{(n)}(x,y)$.
 \end{notation}

\begin{definition}\label{def:QnLocalizefibre}
	For each $n\geq0$, let $\cQ_n$ be the wide sub-$(\infty,1)$-category of $\cC^{(n)}_{/y}$ generated by the maps of zig-zags
	\[
		H\colon Z_n^{\star}\times[1]\longrightarrow\cC
	\]
	with the following property. There is a block of four consecutive vertices such that every component of $H$ outside this block is an equivalence. According to the orientation of the corresponding sub-zig-zag, the restriction of $H$ to this block has one of the following two forms:
	\begin{equation}\label{eq:right}
		\begin{tikzcd}
	x & s & s & t \\
	x & z & z & t
	\arrow["u", from=1-1, to=1-2]
	\arrow["\simeq", swap, from=1-1, to=2-1]
	\arrow["v", from=1-2, to=2-2]
	\arrow["\simeq"', from=1-3, to=1-2]
	\arrow["\alpha v", from=1-3, to=1-4]
	\arrow["v"', from=1-3, to=2-3]
	\arrow[swap, "\simeq ",from=1-4, to=2-4]
	\arrow["{v u }"', from=2-1, to=2-2]
	\arrow["\simeq"', from=2-3, to=2-2]
	\arrow["\alpha", from=2-3, to=2-4]
\end{tikzcd}
	\end{equation}

	For the opposite orientation, it has the form
	\begin{equation}\label{eq:left}
		\begin{tikzcd}
	c & a & a & \ell \\
	c & b & b & \ell{.}
	\arrow[swap,"\simeq ", from=1-1, to=2-1]
	\arrow["{\beta\alpha}"', from=1-2, to=1-1]
	\arrow["\simeq", from=1-2, to=1-3]
	\arrow["\alpha"', from=1-2, to=2-2]
	\arrow["\alpha"', from=1-3, to=2-3]
	\arrow["\gamma"', from=1-4, to=1-3]
	\arrow[swap, "\simeq ", from=1-4, to=2-4]
	\arrow["\beta"', from=2-2, to=2-1]
	\arrow["\simeq", from=2-2, to=2-3]
	\arrow["{ \alpha \gamma}", from=2-4, to=2-3]
\end{tikzcd}
	\end{equation}

 We denote by $\cQ$ the wide sub-$(\infty,1)$-category generated by the union of the various $\cQ_n$'s in $\cC^{\infty}_{/y}$ and by $\cQ_{n,x}$ and $\cQ_x$ the corresponding fibres  over $x \in \cC$.
\end{definition}

\begin{lemma}\label{lem:comparisonQW}
	There is an equivalence of $(\infty,1)$-categories over $\cC$
	\[
		\cC_{/y}^{\infty}[\cQ^{-1}]
		\simeq \BFree^{(1,0)}_{\cL}(\cC_{/y}).
	\]
\end{lemma}
\begin{proof}
	Recall that the construction of the free bifibration identifies $\BFree^{(1,0)}_{\cL}(\cC_{/y})$ with $(\cC_{/y}^{\infty})_{\diamond}$ (cf. the dual of \cref{def:diamond}). Since every stage is an $(\infty,1)$-category, we can identify $(\cC_{/y}^{\infty})_{\diamond}$ with the localization of $\cC_{/y}^{\infty}$ at the $1$-morphisms belonging to $\cW^{1}$ (cf. \cref{def:cwfinal}). Our task is thus to show that $\cW^{1}$ and $\cQ$ generate the same localization.
	
	We first show that each morphism $\omega\in\cQ$ becomes invertible in $\BFree^{(1,0)}_{\cL}(\cC_{/y})$. It suffices to treat the generators. Let $\omega$ be of type \cref{eq:right}, and let $\widehat\omega$ be the map of zig-zags obtained by truncating $\omega$ at the end of the corresponding block. Denote by $\tau$ the map obtained by truncating at the beginning of the block. By construction, $\tau$ is an equivalence, so we may identify its source and target with a single zig-zag $Z$. Under this identification, the image of $\widehat\omega$ in the free bifibration compares two cocartesian lifts of $(\alpha v)u=\alpha(vu)$ with the same source, namely the image of $Z$. Hence this comparison is invertible. The image of $\omega$ in $\BFree^{(1,0)}_{\cL}(\cC_{/y})$ is obtained from that of $\widehat\omega$ by successive cartesian and cocartesian transport functors, and is therefore also invertible. The case of \cref{eq:left} is analogous, using cartesian lifts with the same target.

	For the converse, note that the comparison $1$-morphisms for $T_{k,k+2}$ (cf. \cref{cons:fibrewisejunk}) have the form \cref{eq:left} when $k$ is even and \cref{eq:right} when $k$ is odd, and hence belong to the corresponding graphical classes. Comparisons for longer transitions factor into images of these basic comparisons under the transition functors. The functors $T_{\ell,k}$ preserve our generators for $\cQ$ since they only append identities. Moreover, cartesian transport in even degrees and cocartesian transport over $\cL$ in odd degrees also preserve the two local forms from \cref{def:QnLocalizefibre}. This shows that every morphism in $\cW^{1}$ becomes invertible after localization at $\cQ$, as desired.
\end{proof}

\begin{theorem}\label{thm:mappingcatform}
	Let $\cL\subset\cC$ be a wide sub-$(\infty,1)$-category. Then, for every $x,y\in\cC\soft{\cL}$, there are equivalences of $(\infty,1)$-categories
	\[
		\cC\soft{\cL}(x,y)
		\simeq \colim_{n\in\mathbb{N}}\cC^{(n)}(x,y)[\cQ_{n,x}^{-1}]
		\simeq \cC^{\infty}(x,y)[\cQ_x^{-1}].
	\]
\end{theorem}
\begin{proof}
	By construction, the transition functors $\cC^{(n)}_{/y}\to\cC^{(n+1)}_{/y}$ restrict to functors $\cQ_n\to\cQ_{n+1}$. Compatibility of localization with filtered colimits and cofinality of the even integers therefore give
	\[
		\cC_{/y}^{\infty}[\cQ^{-1}]
		\simeq \colim_{n\in\mathbb N}\cC^{(2n)}_{/y}[\cQ_{2n}^{-1}].
	\]

	Combining \cref{lem:comparisonQW} with the description of mapping categories in terms of the free bifibration, we obtain
	\[
		\cC\soft{\cL}(x,y)
		\simeq \cC_{/y}^{\infty}[\cQ^{-1}]\times_{\cC}\{x\}.
	\]
	We claim that $\cQ_{2n}$ is stable under the cartesian transport: indeed, the transport of a generating morphism has equivalence components outside the same distinguished block and retains one of the two local forms in \cref{def:QnLocalizefibre}. Hence \cite[Theorem 4.7.1]{AHM26}, together with the fact that filtered colimits commute with fibres, gives
	\[
	\begin{aligned}
		\cC\soft{\cL}(x,y)
		&\simeq \colim_n\bigl(\cC_{/y}^{(2n)}[\cQ_{2n}^{-1}]\times_{\cC}\{x\}\bigr) \\
		&\simeq \colim_n\cC^{(2n)}(x,y)[\cQ_{2n,x}^{-1}] \\
		&\simeq \colim_n\cC^{(n)}(x,y)[\cQ_{n,x}^{-1}],
	\end{aligned}
	\]
	where the last equivalence again follows from cofinality. Finally, compatibility of localization with filtered colimits yields
	\[
		\colim_n\cC^{(n)}(x,y)[\cQ_{n,x}^{-1}]
		\simeq \cC^{\infty}(x,y)[\cQ_x^{-1}],
	\]
	which proves the claim.
\end{proof}

When $\cC$ is an $(\infty,1)$-category, the fibration $\BFree^{(1,0)\mhyphen\loc}_{\cL}(\cC_{\upslash c}) \to \cC$ can be computed by geometrically realizing the fibres of $\BFree^{(1,0)}_{\cL}(\cC_{\upslash c}) \to \cC$. Combined with the preceding theorem, this gives the following explicit description of the mapping spaces in $\cC[\cL^{-1}]$. The resulting formula closely resembles the hammock localization introduced by Dwyer and Kan in \cite{DwyerKan1980Calculating}.

\begin{corollary}\label{cor:dwyerkanhammock}
	Let $\cL \subset \cC$ be a wide sub-$(\infty,1)$-category inclusion. Then for every $x,y \in \cC[\cL^{-1}]$ we have equivalences of $\infty$-groupoids
	\[
		\cC[\cL^{-1}](x,y)\simeq  \colim_{\mathbb{N}} |\cC^{(n)}(x,y)| \simeq |\cC^{\infty}(x,y)|.
	\]
\end{corollary}

\subsection{Turning section pairs into reflective localizations}
In this subsection we will address a particular case of how to freely add adjoints, namely, we will freely enforce a section pair to become a (co)reflective localization.

\begin{notation}\label{not:Ret}
	We denote by $\Ret= [2] \coprod_{\{0 \to 2\}} [0]$ the walking retraction, and denote its generating 1-morphisms by $s,r$, subject to $r \circ s =\id$.

	Restriction along $s,r \colon [1] \to \Ret$ induces evaluation functors
	\[
		\ev_s,\ \ev_r \colon \Fun(\Ret,\cC) \longrightarrow \Fun([1],\cC).
	\]
	A functor $f \colon \Ret \to \cC$ is thus the datum of a \emph{section pair} in $\cC$, and we abbreviate
	\[
		s_f \coloneqq \ev_s(f), \qquad r_f \coloneqq \ev_r(f),
	\]
	so that $s_f$ and $r_f$ are 1-morphisms of $\cC$ equipped with an equivalence $r_f \circ s_f \simeq \id$.
\end{notation}

\begin{definition}\label{def:marked-morphisms}
	Let $\cC$ be an $(\infty,2)$-category and let $R,S$ be subspaces of $\Fun(\Ret,\cC)^{\simeq}$. We denote by
	\begin{itemize}
		\item $\cL(R) \subseteq \cC$ the locally full sub-$(\infty,2)$-category generated by the 1-morphisms $r_f$ for $f \in R$;
		\item $\cL(S) \subseteq \cC$ the locally full sub-$(\infty,2)$-category generated by the 1-morphisms $s_g$ for $g \in S$,
	\end{itemize}
	and we set $\cL$ to be the wide locally full sub-$(\infty,2)$-category generated by $\cL(R) \cup \cL(S)$. We write $u^{\mathrm{ra}}$ for the right adjoint of $u \in \cL$ in $\cC\soft{\cL}$, and $\eta_u \colon \id \Rightarrow u^{\mathrm{ra}} \circ u$, $\epsilon_u \colon u \circ u^{\mathrm{ra}} \Rightarrow \id$ for the unit and the counit of the adjunction $u \dashv u^{\mathrm{ra}}$.
\end{definition}

\begin{construction}\label{cons:comparison}
	Let $f \in R$. The equivalence $r_f \circ s_f \simeq \id$ is a candidate counit for an adjunction $r_f \dashv s_f$, and hence induces a comparison 2-morphism in $\cC\soft{\cL}$
	\[
		c_f \colon s_f \xrightarrow{\ \eta_{r_f} \ast s_f\ } r_f^{\mathrm{ra}} \circ r_f \circ s_f \simeq r_f^{\mathrm{ra}} .
	\]
	Dually, let $g \in S$. The equivalence $r_g \circ s_g \simeq \id$ is a candidate unit for an adjunction $s_g \dashv r_g$, and hence induces a comparison 2-morphism in $\cC\soft{\cL}$
	\[
		c_g \colon s_g^{\mathrm{ra}} \simeq r_g \circ s_g \circ s_g^{\mathrm{ra}} \xrightarrow{\ r_g \ast \epsilon_{s_g}\ } r_g .
	\]
\end{construction}

\begin{definition}\label{def:CLRS}
	In the situation of \Cref{def:marked-morphisms}, we denote by $\cC\soft{\cL}_{R,S}$ the localization of $\cC\soft{\cL}$ at the collection of 2-morphisms
	\[
		\{\, c_f \,\}_{f \in R} \ \cup \ \{\, c_g \,\}_{g \in S}
	\]
	of \Cref{cons:comparison}.
\end{definition}

\begin{remark}
	By construction, in $\cC\soft{\cL}_{R,S}$ the freely added right adjoint of $r_f$ is identified with $s_f$ for every $f \in R$, so that $r_f \dashv s_f$ with counit the given equivalence $r_f \circ s_f \simeq \id$; that is, $r_f$ becomes a reflective localization with right adjoint $s_f$. Dually, for every $g \in S$ the freely added right adjoint of $s_g$ is identified with $r_g$, so that $s_g \dashv r_g$ with unit the given equivalence $\id \simeq r_g \circ s_g$; that is, $r_g$ becomes a coreflective localization with left adjoint $s_g$.
\end{remark}

\begin{theorem}\label{thm:ffaddingsections}
	Let $\cC$ be an $(\infty,2)$-category and let $\cC\soft{\cL}_{R,S}$ be as in \cref{def:CLRS}. For every $(\infty,2)$-category $\cX$, restriction along the canonical map $\cC \to \cC\soft{\cL}_{R,S}$ yields a fully faithful functor
	\[
		\Fun(\cC\soft{\cL}_{R,S},\cX) \xrightarrow{} \Fun(\cC,\cX),
	\]
	whose essential image consists of those functors $F \colon \cC \to \cX$ sending each $r_f \circ s_f \xrightarrow{\simeq} \id$ in $R$ to the counit of an adjunction, and each $\id \xrightarrow{\simeq} r_g \circ s_g $ in $S$ to the unit of an adjunction.
\end{theorem}
\begin{proof}
	The universal property of $\cC\soft{\cL}_{R,S}$ implies that the restriction functor is locally fully-faithful, with objects as in the statement, but also requiring 1-morphisms to be adjointable natural transformations. By \cref{BCeasy,rem:BCeasy} this final condition is superfluous and thus the result.
\end{proof}

\begin{definition}\label{def:representableRS}
	Let $\cC$, $R$, and $S$ be as in \cref{def:marked-morphisms}. For $c \in \cC$, define $\BFree^{(1,0)}_{R,S}(\cC_{\upslash c}) \to \cC$ as the localization of $\BFree^{(1,0)}_{\cL}(\cC_{\upslash c}) \to \cC$ at the fibre-wise morphisms (and their cartesian transports) corresponding to the comparison maps of \cref{cons:comparison}.
\end{definition}

The proof of the next theorem is completely analogous to that of \cref{thm:addingadjoints}, and we therefore omit it.

\begin{theorem}\label{thm:fibmodelRS}
	Let $\cC$ be an $(\infty,2)$-category and let $\cC\soft{\cL}_{R,S}$ be as in \cref{def:CLRS}. Then $\cC\soft{\cL}_{R,S}$ is equivalent to the full sub-$(\infty,2)$-category of $\Fib^{(1,0)\mhyphen\bifib}_{/(\cC,\cL)}$ spanned by the objects $\BFree^{(1,0)}_{R,S}(\cC_{\upslash c}) \to \cC$, for $c \in \cC$.
\end{theorem}

\section{Applications}

\subsection{The walking adjunction}\label{sec:walkingadjunction}
The goal of this section is to describe $\cC\soft{\cL}$ when $\cL=\cC=[1]$. This recovers a celebrated theorem of Riehl and Verity \cite{RiehlVerity2016FormalTheoryMonads}, which asserts that every left adjoint in an $(\infty,2)$-category extends, essentially uniquely, to a functor from the walking adjunction. In our framework, this takes the following form:

\begin{theorem}\label{thm:RVmi}
	Let $\cC=\cL=[1]$, and let $\varphi \colon [1] \to \Adj$ be the functor selecting the universal left adjoint. Then $\Adj$ corepresents the functor $\Fun(\cC,-)_{\cL}$ of \cref{def:leftlocalization1}, with the equivalence induced by restriction along $\varphi$.
\end{theorem}

\begin{corollary}\label{cor:esssurj}
	The $(\infty,2)$-category $\cC\soft{\cL}$ can be obtained as an iterated pushout along the morphism $[1] \to \Adj$. 
\end{corollary}

\begin{proposition}\label{prop:detecsoft}
	Let $\cL \subset \cC$ with $\cC$ an $(\infty,1)$-category, and let $\varphi \colon \cC \to \cD$ be a functor of $(\infty,2)$-categories. Suppose the following conditions are met:
	\begin{itemize}
		\item $\varphi$ is essentially surjective.
		\item $\varphi$ sends every morphism in $\cL$ to a left adjoint in $\cD$.
		\item For every $c \in \cC$ and $\varphi(c)\simeq d$, the canonical square
		\[
			\begin{tikzcd}
				\BFree^{(1,0)}_{\cL}(\cC_{\upslash c}) \arrow[d] \arrow[r] & \cD_{ \upslash d} \arrow[d] \\
				\cC \arrow[r,"\varphi"] & \cD
			\end{tikzcd}
		\]
		is a pullback.
	\end{itemize}
	Then the induced functor $\cC\soft{\cL} \to \cD$ is an equivalence of $(\infty,2)$-categories.
\end{proposition}
\begin{proof}
	Since the composite $\cC \to \cC\soft{\cL} \xrightarrow{\overline{\varphi}} \cD$ recovers $\varphi$, which is essentially surjective by assumption, it follows that $\overline{\varphi}$ is also essentially surjective. The final condition implies that restriction along $\overline{\varphi}$ preserves representable functors giving the result.
\end{proof}

To prove \cref{thm:RVmi}, we verify that $\varphi \colon [1] \to \Adj$ satisfies the conditions of \cref{prop:detecsoft}; establishing this will occupy the remainder of this section.

\begin{notation}
	For $j \in [1]$, let $\pi \colon [1]_{/j} \to [1]$ denote the canonical projection, and let $D_j \to [1]$ be the filtered colimit of the (dual of the) functor $T$ of \cref{cons:Tbifib}, taken relative to $\pi$. Write $D(i,j)$ for the fibre of $D_j$ over $i \in [1]$.
\end{notation}

\begin{remark}\label{rem:overlinemu}
	Observe that the canonical commutative square,
	\[
		\begin{tikzcd}
			{[1]_{/j}} \arrow[d] \arrow[r] & \Adj_{\upslash \varphi(j)} \arrow[d] \\
			{[1]} \arrow[r,"\varphi"] & \Adj
		\end{tikzcd}
	\]
	induces a morphism $\overline{\mu}_j \colon D_j \to \varphi^* \Adj_{\upslash \varphi(j)}$. Passing to fibres we obtain for every $i,j \in [1]$ a functor
	\[
		\overline{\mu}_{i,j} \colon D(i,j) \to \Adj(\varphi(i),\varphi(j)).
	\]
\end{remark}

\begin{construction}\label{cons:mucombinatorial}
	For $n \geq 0$, we consider the following pullback diagram
	\[
		\begin{tikzcd}
			D(i,j)^{(n)}  \arrow[r] \arrow[d] & \Fun(Z_n^{\star},[1]) \arrow[d,"\ev_{-1} \times \ev_n"] \\
			{[0]} \arrow[r,"j \times i "] & {[1]\times [1]}
		\end{tikzcd}
	\]
	with $\colim_{n \geq 0}D(i,j)^{(n)} \simeq D(i,j)$. We will construct functors $\overline{\mu}^{(n)}_{i,j} \colon D(i,j)^{(n)} \to \Adj(\varphi(i),\varphi(j))$, whose filtered colimit recovers $\overline{\mu}_{i,j}$.

	For every $f \colon Z_n^{\star} \to [1]$ we form the pullback along $i \colon \{1\} \to [1]$ and denote the resulting sub-poset by $f^{-1}(1)$; write $\pi_0(f^{-1}(1))$ for its set of connected components. Since the underlying graph of $Z_n^{\star}$ is a path, its pairwise disjoint connected subgraphs are linearly ordered by their position along it; applied to the components of $f^{-1}(1)$, this endows $\pi_0(f^{-1}(1))$ with a linear order, which we read off from the initial vertex $-1$. Moreover, since $i$ is a cocartesian fibration, the assignment $f \mapsto f^{-1}(1)$ is functorial with respect to natural transformations, and one checks that the induced maps on connected components are monotone. We have thus constructed a functor $\overline{\mu}_{i,j}^{(n)}$ valued in $\Delta_{+}$. Moreover, we have that:
	\begin{itemize}
		\item if $i=1$, $j=0$ this map factors through $\Delta_{\infty}$.
		\item if $i=0$ and $j=1$ this map factors through $\Delta_{-\infty}$.
		\item if $i=j=1$ this map factors through $\Delta_{\mathsf{int}}$.
	\end{itemize}
\end{construction}

The functor $\overline{\mu}_{i,j}$ descends to the localization associated to $\BFree^{(1,0)}_{[1]}([1]_{/j})\times_{[1]}\{i\}$, yielding a functor
\begin{equation}\label{eq:mu}
	\mu_{i,j} \colon \BFree^{(1,0)}_{[1]}([1]_{/j})\times_{[1]}\{i\} \to \Adj(\varphi(i),\varphi(j)).
\end{equation}
Our goal is to show that $\mu_{i,j}$ is an equivalence. 

\begin{lemma}\label{lem:pi0iso}
	Let $u$ be a morphism in $D(i,j)$ such that $\overline{\mu}_{i,j}(u)$ is invertible. Then the image of $u$ in $\BFree^{(1,0)}_{[1]}([1]_{/j})\times_{[1]}\{i\}$ is invertible. 
\end{lemma}
\begin{proof}
	Specializing the generators of \cref{def:QnLocalizefibre}, we see that the localization is generated by morphisms with the following local forms:
	\[
		\begin{tikzcd}
	0 & 0 & 0 & 1 \\
	0 & 1 & 1 & 1{,}
	\arrow[ from=1-1, to=1-2]
	\arrow["\simeq", swap, from=1-1, to=2-1]
	\arrow[ from=1-2, to=2-2]
	\arrow["\simeq"', from=1-3, to=1-2]
	\arrow[ from=1-3, to=1-4]
	\arrow[ from=1-3, to=2-3]
	\arrow[swap, "\simeq ",from=1-4, to=2-4]
	\arrow[ from=2-1, to=2-2]
	\arrow["\simeq"', from=2-3, to=2-2]
	\arrow[ from=2-3, to=2-4]
\end{tikzcd} \quad \quad  	\begin{tikzcd}
	1 & 0 & 0 & 0 \\
	1 & 1 & 1 & 0{.}
	\arrow[swap,"\simeq ", from=1-1, to=2-1]
	\arrow[ from=1-2, to=1-1]
	\arrow["\simeq", from=1-2, to=1-3]
	\arrow[ from=1-2, to=2-2]
	\arrow[ from=1-3, to=2-3]
	\arrow[ from=1-4, to=1-3]
	\arrow[swap, "\simeq ", from=1-4, to=2-4]
	\arrow[ from=2-2, to=2-1]
	\arrow["\simeq", from=2-2, to=2-3]
	\arrow[ from=2-4, to=2-3]
\end{tikzcd}
	\]
	Let $u_0$ and $u_1$ denote the source and target zig-zags of $u$, respectively. Since $u$ is pointwise given by morphisms $0\to1$ or identities, we have an inclusion $u_0^{-1}(1)\subseteq u_1^{-1}(1)$. Moreover, the invertibility of $\overline{\mu}_{i,j}(u)$ implies that this inclusion induces a bijection on connected components. Thus every component $I$ of $u_1^{-1}(1)$ contains a unique component $J$ of $u_0^{-1}(1)$. If $u_1$ is constant, this bijection together with the endpoint conditions forces $u_0=u_1$, so there is nothing to prove.

	We may therefore work independently on each pair $J\subseteq I$. Whenever a boundary of $J$ differs from the corresponding boundary of $I$, the alternating orientation of the zig-zag forces the distance between them to be even. Moving that boundary outward two vertices at a time expresses the inclusion $J\subseteq I$ as a composite of morphisms having one of the two local forms displayed above. The right-boundary and endpoint cases are obtained in the same way, or dually by reversing the zig-zag. Performing this construction for every component factors $u$ as a composite of generators of the localization, and hence its image is invertible.
\end{proof}

The remaining input is a concrete geometric model for $\overline{\mu}_{i,j}$.

\begin{definition}\label{def:Kabs}
	Let $\cK$ be the set of closed subsets of $\mathbb{R}_{\geq 0}$ with finitely many connected components, ordered by inclusion. Each non-empty proper subset $L \in \cK$ is a finite disjoint union of compact intervals (possibly singletons) and at most one closed ray $[a,\infty)$. We denote by $\cK^{\cell}$ the subposet consisting of those objects whose finite endpoints belong to $\mathbb{N}$.
\end{definition}

\begin{definition}\label{def:Ksepinfty}
	We define several subposets of $\cK^{\cell}$ (see \cref{def:Kabs}) as follows:
	\begin{itemize}
		\item We denote by $\cK^{\cell}_+$ the subposet on those objects $L$ such that $0 \notin L$ and $L$ contains no closed ray.
		\item We denote by $\cK^{\cell}_{-\infty}$ the subposet on those objects $L$ such that $0 \in L$ and $L$ contains no closed ray.
		\item We denote by $\cK^{\cell}_{\infty}$ the subposet on those objects $L$ such that $0 \notin L$ and $L$ contains a closed ray.
		\item We denote by $\cK^{\cell}_{\intv}$ the subposet on those objects $L$ such that $0 \in L$ and $L$ contains a closed ray.
	\end{itemize}
\end{definition}

\begin{lemma}\label{lem:posetgeom}
	We have the following equivalences of posets: $D(1,0)\simeq \cK^{\cell}_{\infty}$, $D(0,1)=\cK^{\cell}_{-\infty}$, $D(0,0) \simeq \cK^{\cell}_+$ and $D(1,1)=\cK^{\cell}_{\intv}$.
\end{lemma}
\begin{proof}
	We only verify the case of $D(1,1)$, the remaining cases being similar. Given a zig-zag $f \colon Z_n^{\star} \to [1]$ with $f(-1)=f(n)=1$, assign to each $j \in Z_n^{\star}$ the point $(j+1)/2 \in \mathbb{R}_{\geq 0}$. To produce an element of $\cK^{\cell}_{\intv}$, map each connected component of $1$'s to the interval determined by its endpoints, with one exception: writing $j_0$ for the smallest index with $f(k)=1$ for all $j_0 \leq k \leq n$, we instead associate $[(j_0+1)/2,\infty)$ to that final connected component. This assignment yields a map $D(1,1)^{(n)} \to \cK^{\cell}_{\intv}$ compatible with the transition maps; passing to filtered colimits, we obtain a map $D(1,1) \to \cK^{\cell}_{\intv}$, which is easily verified to be an isomorphism.
\end{proof}

\begin{figure}[ht]
\centering
\begin{tikzcd}[row sep=1.5cm]
\begin{tikzpicture}[
    baseline=(current bounding box.center),
    x=0.72cm,
    y=0.85cm,
    one/.style={
        circle,
        draw=none,
        fill=BrickRed,
        fill opacity=0.10,
        text opacity=1,
        text=BrickRed,
        minimum size=4.8mm,
        inner sep=0pt,
        font=\small\bfseries
    },
    zero/.style={
        text=black!78,
        inner sep=1pt,
        font=\small\bfseries
    },
    ordinary edge/.style={
        ->,
        thick,
        black!65,
        shorten <=3pt,
        shorten >=3pt
    },
    red edge/.style={
        ->,
        very thick,
        BrickRed,
        shorten <=2pt,
        shorten >=2pt
    }
]
    \node[one]  (v0)  at (0,1)  {$1$};
    \node[zero] (v1)  at (1,0)  {$0$};
    \node[zero] (v2)  at (2,1)  {$0$};
    \node[zero] (v3)  at (3,0)  {$0$};
    \node[one]  (v4)  at (4,1)  {$1$};
    \node[zero] (v5)  at (5,0)  {$0$};
    \node[zero] (v6)  at (6,1)  {$0$};
    \node[zero] (v7)  at (7,0)  {$0$};
    \node[one]  (v8)  at (8,1)  {$1$};
    \node[one]  (v9)  at (9,0)  {$1$};
    \node[one]  (v10) at (10,1) {$1$};

    \draw[ordinary edge] (v1) -- (v0);
    \draw[ordinary edge] (v1) -- (v2);
    \draw[ordinary edge] (v3) -- (v2);
    \draw[ordinary edge] (v3) -- (v4);
    \draw[ordinary edge] (v5) -- (v4);
    \draw[ordinary edge] (v5) -- (v6);
    \draw[ordinary edge] (v7) -- (v6);
    \draw[ordinary edge] (v7) -- (v8);

    \draw[red edge] (v9) -- (v8);
    \draw[red edge] (v9) -- (v10);
\end{tikzpicture}
\arrow[
    d,
    black!68,
    thick,
    shorten <=7pt,
    shorten >=7pt,
    "{j\longmapsto (j+1)/2}"',
    font=\small
]
\\
\begin{tikzpicture}[
    baseline=(current bounding box.center),
    x=1.35cm,
    y=0.7cm
]
    \draw[->, very thick, black!68]
        (-0.25,0) -- (5.75,0)
        node[right, text=black!78] {$\mathbb R_{\geq 0}$};

    \foreach \x in {0,...,5} {
        \draw[black!60] (\x,0.10) -- (\x,-0.10);
        \node[below=4pt, text=black!75, font=\scriptsize]
            at (\x,0) {$\x$};
    }

    \fill[BrickRed] (0,0) circle (3pt);
    \fill[BrickRed] (2,0) circle (3pt);
    \draw[BrickRed, line width=5pt, line cap=round]
        (4,0) -- (5.55,0);

    \node[above=7pt, text=BrickRed, font=\small]
        at (0,0) {$\{0\}$};
    \node[above=7pt, text=BrickRed, font=\small]
        at (2,0) {$\{2\}$};
    \node[above=7pt, text=BrickRed, font=\small]
        at (4.75,0) {$[4,\infty)$};
\end{tikzpicture}
\end{tikzcd}

\caption{An object of $D(1,1)$ and its corresponding element
$\{0\}\cup\{2\}\cup[4,\infty)$ of $\cK^{\cell}_{\intv}$.
The connected components of vertices labelled $1$ determine the red
components of the cellular subset.}
\label{fig:zigzag-cellular-subset}
\end{figure}
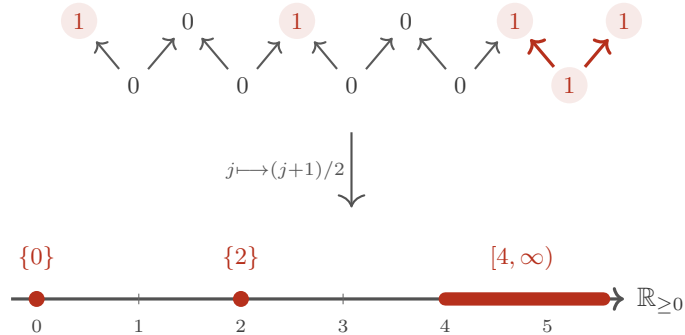

\begin{remark}\label{rem:posetgeom}
	Under the identifications of \cref{lem:posetgeom}, the map $\overline{\mu}_{i,j}$ is identified with the functor $\pi_0$ sending $L \subseteq \mathbb{R}_{\geq 0}$ to its set of connected components, ordered from left to right.
\end{remark}

\begin{definition}
	Let $\star \in \{+, -\infty,\infty, \intv \}$ and let $\pi_0 \colon \cK_{\star}^{\cell}\to \Delta_{\star}$ be the functor of \cref{rem:posetgeom}. Denote by $\cW_{\star} \subset \cK_{\star}^{\cell}$ the wide subposet on those morphisms which are inverted by $\pi_0$. 
\end{definition}

\begin{proposition}\label{prop:assemblpart1}
	Let $\star \in \{+, -\infty,\infty, \intv \}$. Then for every $[n] \in \Delta$, postcomposition with $\pi_0$ induces an equivalence
	\[
		|\Fun_{\cW_\star}([n],\cK_\star^{\cell})| \xrightarrow{\simeq} \Fun([n],\Delta_\star)^{\simeq}.
	\]
	Here the source denotes the geometric realization of the wide subposet of $\Fun([n],\cK_\star^{\cell})$ whose morphisms are the natural transformations pointwise in $\cW_{\star}$.
\end{proposition}
\begin{proof}
	Since the target is discrete, it will be enough to show that the poset $\cR_{\sigma}$ defined by the pullback square
	\[
		\begin{tikzcd}
			\cR_{\sigma} \arrow[r] \arrow[d] & \Fun_{\cW_\star}({[n]},\cK_\star^{\cell}) \arrow[d] \\
			{[0]} \arrow[r,"\sigma"] & \Fun([n],\Delta_\star)^{\simeq}
		\end{tikzcd}
	\]
	is contractible for every $\sigma$. Write $\sigma(j)=[\ell_j]$ and $u_j=\sigma(j\to j+1)$. An object $\vec K\in\cR_\sigma$ is a chain $K_0\subseteq\cdots\subseteq K_n$ where each level has connected components $K_j^\alpha$ indexed by $\alpha\in[\ell_j]$, with $K_j^\alpha\subseteq K_{j+1}^{u_j(\alpha)}$.

    Suppose that $\star=+$. Then if $\sigma$ is constant on the initial object it follows that $\cR_\sigma=[0]$ which is contractible. More generally, let $i \in [n]$ be the smallest index such that $\sigma(i)\neq \emptyset$ and denote by $\widehat{\sigma}=\sigma_{|[i,n]}$. Then it follows that $\cR_{\sigma}\simeq \cR_{\widehat{\sigma}}$ so we may assume without loss of generality that $\sigma$ does not hit the initial object. A similar argument shows the dual statement for $\star=\intv$ and the terminal object.

	For every $\sigma$, we can find an object $\vec P\in\cR_\sigma$ such that:
\begin{itemize}
    \item Each component $P_j^\alpha$ has distinct endpoints.
    \item $P_j^\alpha$ and $P_{j+1}^{u_j(\alpha)}$ have no common endpoint, except for $0$ on the first component when $\star\in\{-\infty,\intv\}$, and $\infty$ on the last component when $\star\in\{\infty,\intv\}$.
\end{itemize}
Here an unbounded component is regarded as having right endpoint $\infty$. 

Write $K_j^\alpha=[a(K)_j^\alpha,b(K)_j^\alpha]$, where the endpoints take values in $\mathbb{N}\cup\{\infty\}$. For $\vec K\in\cR_\sigma$, let $I(\vec K)$ be the set of labelled endpoints $a(K)_j^\alpha,b(K)_j^\alpha$, for $0\leq j\leq n$ and $\alpha\in[\ell_j]$, omitting the left-endpoint labels forced to be $0$ for $\star \in \{-\infty,\intv\}$ and the right-endpoint labels forced to be $\infty$ for $\star \in \{\infty,\intv\}$. These are regarded as distinct labels even when their numerical values coincide. By our choice of $\vec P$, the values of the endpoints in $I(\vec P)$ are pairwise distinct and hence inherit a linear order from $\mathbb{R}_{\geq 0}$. By construction we have a canonical isomorphism $I(\vec{P})\simeq I(\vec{K})$ sending $a(P)^{\alpha}_j \mapsto a(K)^{\alpha}_j$ and similarly for the right endpoints, which we use to transfer the linear order of $I(\vec{P})$ to any other $I(\vec{K})$. Note that the order on $I(\vec{K})$ is also induced from $\mathbb{R}_{\geq 0}$ and that we use $I(\vec{P})$ only to disambiguate when two endpoints coincide.

For $i \in I(\vec{P})$, denote by $\cR_\sigma^{\leq i}$ the full subposet on those chains such that for every $j \leq i$, the endpoint labelled by $j$ in $I(\vec{K})$ is less than or equal to the corresponding endpoint in $I(\vec{P})$. This yields a filtration
\[
	\cL_\sigma=\cR_{\sigma}^{\leq \max(I(\vec{P}))} \subseteq \cdots\subseteq \cR_{\sigma}^{\leq \min(I(\vec{P}))} \subseteq \cR_\sigma.
\]
We will show that each step induces a homotopy equivalence on geometric realizations.

Numbering the elements of $I(\vec{K})$ increasingly from $1$ to $N$, set $\cR_\sigma^{\leq 0}=\cR_\sigma$. Let $\vec K\in\cR_\sigma^{\leq s}$ for $0\leq s<N$, and consider the $(s+1)$-st endpoint $x(K)_j^\alpha\in I(\vec K)$, where $x\in\{a,b\}$.

We claim that replacing this endpoint with $\min(x(K)_j^\alpha,x(P)_j^\alpha)$, leaving all other endpoints unchanged, yields an object $r_s(\vec{K})\in\cR_\sigma^{\leq s+1}$. We assume that $x=a$, the other case being analogous. Moreover, we can assume that $a(P)_j^\alpha < a(K)_j^\alpha$ since otherwise the claim is trivially true. We make the following observations
\begin{itemize}
	\item If $x=a$ we have that $K_j \subset r_s(\vec{K})_j$ since we are possibly expanding one interval.
	\item For $0\leq i\leq n$ with $i \neq j$, we have that $r_s(\vec K)_i=K_i$. In particular, if $0\leq i<j$ then $r_s(\vec K)_i \subset r_s(\vec K)_j$, and if $j+1\leq i<n$, we have that $r_s(\vec K)_i \subset r_s(\vec K)_{i+1}$.
\end{itemize} 
If $j<n$, it remains to check that $r_s(\vec{K})_j \subset r_s(\vec{K})_{j+1}$; for $j=n$, there is no further nesting condition. In the former case, this amounts to showing that
\[
	[a(P)_j^\alpha,b(K)_j^{\alpha}] \subset [a(K)_{j+1}^{u_j(\alpha)},b(K)_{j+1}^{u_j(\alpha)}].
\]
By construction we have that $a(P)_{j+1}^{u_j(\alpha)} \leq a(P)_{j}^{\alpha}$ which implies by our assumptions that $ a(K)_{j+1}^{u_j(\alpha)} \leq  a(P)_{j+1}^{u_j(\alpha)}$. Therefore we have
\[
	[a(P)_{j}^{\alpha},b(K)_j^{\alpha}]\subset [a(P)_{j+1}^{u_j(\alpha)},b(K)_{j+1}^{u_j(\alpha)}]\subseteq [a(K)_{j+1}^{u_j(\alpha)},b(K)_{j+1}^{u_j(\alpha)}],
\]
as desired. To conclude that $r_s(\vec{K})\in \cR_\sigma$ we must show that the canonical map $K_j \subset r_s(\vec{K})_j$ is an equivalence on connected components. For $\alpha=0$, there is no preceding component. For $\alpha>0$, the only possible obstruction is that $[a(P)_{j}^{\alpha},b(K)_j^{\alpha}]$ intersects $[a(K)_j^{\alpha-1},b(K)_j^{\alpha-1}]$. Note that since $b(K)_{j}^{\alpha-1} \leq b(P)_{j}^{\alpha-1}<a(P)_{j}^{\alpha}$ this cannot happen.

 By monotonicity of the minimum it follows that the assignment above defines a functor
\[
	r_{s}\colon \cR_{\sigma}^{\leq s}  \to \cR_{\sigma}^{\leq s+1},
\]
which is a retraction of the canonical inclusion $\iota_s$. Moreover, according to whether $x=a$ or $x=b$, the corresponding interval expands or shrinks, giving a natural transformation $\id\to\iota_s\circ r_s$ or $\iota_s\circ r_s\to\id$, respectively, and hence the desired homotopy equivalence.

To finish the proof, for $1\leq j\leq N+1$ we define a dual filtration $\cL_\sigma^{\geq j}\subseteq\cL_\sigma$ by requiring, for every $j\leq s\leq N$, the $s$-th endpoint in $I(\vec K)$ to be greater than or equal to its counterpart in $I(\vec P)$. Thus $\cL_\sigma^{\geq N+1}=\cL_\sigma$, and since the reverse inequalities already hold in $\cL_\sigma$, we have $\cL_\sigma^{\geq1}=\{\vec P\}$. Replacing minima with maxima and treating endpoints in decreasing order, the same argument shows that each inclusion induces a homotopy equivalence on geometric realizations. Thus $\cR_\sigma$ is contractible, as required.
\end{proof}
\medskip

\begin{proof}[Proof of \cref{thm:RVmi}]

  We need to verify that the third condition in \cref{prop:detecsoft} holds. Combining \cref{thm:mappingcatform,rem:overlinemu,lem:pi0iso}, it will be enough to show that the map
  \[
 	\overline{\mu}_{i,j} \colon D(i,j) \to \Adj(\varphi(i),\varphi(j)),
   \]
   is a localization at those morphisms inverted by $\overline{\mu}_{i,j}$. \cref{lem:posetgeom} identifies $D(i,j)$ with $\cK^{\cell}_\star$ for $\star \in \{+,-\infty,\infty,\intv\}$, under which $\overline{\mu}_{i,j}$ becomes $\pi_0 \colon \cK^{\cell}_\star \to \Delta_\star$ by \cref{rem:posetgeom}.

   By \cref{prop:assemblpart1}, $\pi_0$ induces a degreewise equivalence from the Rezk nerve of $(\cK^{\cell}_\star,\cW_\star)$ to the complete Segal nerve of $\Delta_\star$. By \cite[Theorem 3.8]{MazelGee2019UniversalityRezkNerve}, the complete Segal space associated to the former presents $\cK^{\cell}_\star[\cW_\star^{-1}]$. Thus $\pi_0$ is the required localization, and each $\mu_{i,j}$ is an equivalence.
\end{proof}

\begin{corollary}\label{cor:adfjffuniv}
	Let $\cC=\Ret$, $R=\{\id \colon \Ret \to \Ret\}$ and $S=\emptyset$. Then $\cC\soft{\cL}_{R,S} \simeq \Adjffr$. Dually, if $R=\emptyset$ and $S=\{\id \colon \Ret \to \Ret\}$, we have $\cC\soft{\cL}_{R,S}=\Adjffl$.
\end{corollary}
\begin{proof}
	We treat the case $R=\{\id \colon \Ret \to \Ret\}$ and $S=\emptyset$; the other case is analogous. By \cref{thm:RVmi}, we may identify $\cC\soft{\cL}_{R,S}$ with the localization of $\Adj$ at the 2-morphism $[1] \to [0]$ in $\Adj(-,-)=\Delta_{\intv}$. The inclusion $[1] \to \Adjffr$ selecting the universal left adjoint extends essentially uniquely, again by \cref{thm:RVmi}, to a functor $\Adj \to \Adjffr$. It is a routine computation --where one checks that this map induces localizations at the level of mapping categories-- to verify that this functor exhibits $\Adjffr$ as the desired localization.
\end{proof}

\subsection{The universal property of the simplex 2-category}

\begin{definition}\label{def:ddelta}
	The \emph{simplex 2-category} $\ddelta$ is the full sub-$(\infty,2)$-category of $\Cat_{(\infty,1)}$ spanned by the nonempty finite linearly ordered sets $[n] = \{0 < 1 < \cdots < n\}$, $n \geq 0$, regarded as posets. Explicitly, $\ddelta$ has objects $[n]$ for $n \geq 0$; a 1-morphism $[n] \to [m]$ is an order-preserving map; and, given order-preserving maps $f,g \colon [n] \to [m]$, there is a (unique) 2-morphism $f \rightarrow g$ if and only if $f(i) \leq g(i)$ for every $i \in [n]$. In particular, each mapping category $\ddelta([n],[m])$ is a poset, so that $\ddelta$ is a $(2,2)$-category.
\end{definition}

\begin{notation}\label{not:inequality}
	Since $\ddelta$ is a $(2,2)$-category enriched in posets, we write $f \leq g$ for a 2-morphism $f \to g$ in $\ddelta$.
\end{notation}

\begin{remark}\label{rem:generators}
	For $n \geq 0$ and $0 \leq i \leq n$, write $d_i,d_{i+1} \colon [n] \to [n+1]$ for the order-preserving injections skipping the values $i$ and $i+1$ respectively, and $s_i \colon [n+1] \to [n]$ for the order-preserving surjection repeating the value $i$. These maps assemble into a string of adjunctions
	\[
		d_{i+1} \dashv s_i \dashv d_i
	\]
	in $\ddelta$: that is, $s_i$ is simultaneously right adjoint to $d_{i+1}$ and left adjoint to $d_i$, where the simplicial identities $\id \xrightarrow{\simeq} s_i d_{i+1}$ and $s_i d_i \xrightarrow{\simeq} \id$ witness the unit and counit of the respective adjunctions.
\end{remark}


\begin{definition}\label{def:representingdelta}
	For $n \geq 0$ and $0 \leq i \leq n$ (\cref{rem:generators}), let $R \subset \Fun(\Ret,\Delta)^{\simeq}$ be the subspace of those section pairs $f$ (\cref{not:Ret}) with $s_f = d_i$ and $r_f = s_i$, and let $S \subset \Fun(\Ret,\Delta)^{\simeq}$ be the subspace of those section pairs $g$  with $s_g = d_{i+1}$ and $r_g = s_i$ . We set
	\[
		\ddelta_{\diamond} \coloneqq \Delta\soft{\cL}_{R,S},
	\]
	as in \cref{def:CLRS}.
\end{definition}

\begin{remark}\label{rem:Lintrinsinc}
	Note that by construction the class $\cL\subset \Delta$ from \cref{def:representingdelta} can be characterised as those $1$-morphisms  $f \colon [n]\to [m]$ such that $f(0)=0$.
\end{remark}

\begin{remark}\label{rem:representingdelta}
	By \cref{thm:ffaddingsections} applied with $\cC=\Delta$, for every $(\infty,2)$-category $\cX$, restriction along $\Delta \to \ddelta_{\diamond}$ identifies $\Fun(\ddelta_{\diamond},\cX)$ with the full sub-$(\infty,2)$-category of $\Fun(\Delta,\cX)$ on those functors $D$ sending each invertible 2-morphism $\id \xrightarrow{\simeq} s_id_{i+1}$ to the unit, and each $s_id_i \xrightarrow{\simeq} \id$ to the counit, of an adjunction.
\end{remark}
The universal property of $\ddelta_{\diamond}$ yields a canonical morphism,
	\begin{equation}
		\varphi \colon \ddelta_{\diamond} \to \ddelta.
	\end{equation}
We can now state the main result of this section.
\begin{theorem}\label{thm:2simplex}
	The functor $\varphi \colon \ddelta_{\diamond} \to \ddelta$ is an equivalence of $(\infty,2)$-categories.
\end{theorem}

\begin{remark}\label{rem:redremark}
	We note that the composite functor $\Delta \to \ddelta_\diamond \xrightarrow{\varphi} \ddelta,$
induces an equivalence after passage to underlying $(\infty,1)$-categories. This implies that $\varphi$ is essentially surjective on objects and 1-morphisms. Therefore, for the proof of the theorem, we will only need to show that $\varphi$ induces a fully faithful functor on mapping $(\infty,1)$-categories.
\end{remark}

\begin{lemma}\label{lem:esssurj1morph}
	The functor $j \colon \Delta \to \ddelta_{\diamond}$ is essentially surjective on objects and 1-morphisms.
\end{lemma}
\begin{proof}
	This follows from \cref{prop:essurjcobase} and the fact that $\ddelta_{\diamond}$ is obtained as an iterated pushout along the morphisms $\Ret \to \Adjffr$ and $\Ret \to \Adjffl$.
\end{proof}

\begin{lemma}\label{lem:keyreduction2simplex}
	Suppose that for every $n\geq 0$ we have that $\Map_{\ddelta_{\diamond}(n,n)}(\id_n,\id_n)\simeq *$. Then $\varphi$ is fully-faithful.
\end{lemma}
\begin{proof}
	As explained in \cref{rem:redremark}, the question reduces to showing that $\varphi$ yields a fully-faithful functor on mapping categories. Invoking \cref{lem:esssurj1morph}, it will be enough to show that for 1-morphisms $f,g \in \Delta$, the mapping spaces $\Map_{\ddelta_{\diamond}(n,m)}(f,g)$ are contractible whenever $\Map_{\ddelta(n,m)}(f,g)$ is inhabited.

	To achieve this, we work inductively. Before setting up the general induction, we make three observations:
	\begin{enumerate}
		\item Suppose that $f=d_{i+1}\hat{f}$. Then, using the adjunction $d_{i+1}\dashv s_i$, we obtain an equivalence of mapping spaces
	\[
		\Map_{\ddelta_{\diamond}(n,m)}(f,g) \simeq \Map_{\ddelta_{\diamond}(n,m-1)}(\hat{f},s_ig).
	\]
	   \item If $f=d_0\hat{f}$, it follows that $0<f(0)\leq g(0)$, which shows that $g=d_0 \hat{g}$. Since $d_0$ is a fully-faithful 1-morphism, we obtain an equivalence of mapping spaces
	\[
		\Map_{\ddelta_{\diamond}(n,m-1)}(\hat{f},\hat{g})\xrightarrow{\simeq}\Map_{\ddelta_{\diamond}(n,m)}(f,g).
	\]
	 \item If $f=\hat{f}s_i$, then we can use the adjunction $d_{i+1}\dashv s_i$ to produce an equivalence of mapping spaces
	 \[
		\Map_{\ddelta_{\diamond}(n,m)}(f,g)\xrightarrow{\simeq}\Map_{\ddelta_{\diamond}(n-1,m)}(\hat{f},gd_{i+1}).
	 \]
	\end{enumerate}
	We are ready to prove the claim by induction on the pair $([n],[m])$. Note that the case $n=m=0$ is covered by the assumptions in the statement. If $f$ satisfies any of the conditions (1)--(3) above, the claim follows from the induction hypothesis. Otherwise, we must have $f=\id_n$. If $g$ is injective, then $g=\id_n$, and the claim follows from our assumptions. Otherwise, we must have $g=\hat{g}s_i$. Since we have an adjunction $s_i \dashv d_i$, we obtain an equivalence of mapping spaces
	\[
		\Map_{\ddelta_{\diamond}(n,n)}(\id_n,g) \xrightarrow{\simeq}\Map_{\ddelta_{\diamond}(n-1,n)}(d_i,\hat{g}).
	\]
	The claim now follows from the induction hypothesis.
\end{proof}

\begin{notation}\label{not:Qnninfty}
	Let $Q/n \to \Delta$ denote the filtered colimit of (the dual of) the functor $T$ in \cref{cons:Tbifib}. For every $[m]\in \Delta$ denote its fibre by $Q(m,n)$. Given $f \in \ddelta(m,n)$ we denote by $Q(m,n)_f$ the fibre over $f$ of the canonical functor $Q(m,n) \to \ddelta(m,n)$.
\end{notation}

\begin{remark}\label{rem:univadjointable}
	An important step in our proof will be to show that every morphism in $Q(n,n)_{\id}$ becomes invertible in $\BFree^{(1,0)}_{R,S}(\Delta_{/n})\times_\Delta\{n\}$ (cf. \cref{def:representableRS}). An object represented by a map $Z_k^\star\to\Delta$ determines a morphism in $\ddelta_\diamond$ by reading the zig-zag from right to left, replacing each right-pointing morphism by its right adjoint, and composing. Likewise, if $H\colon Z_k^\star\times[1]\to\Delta$ represents a morphism, the induced 2-morphism in $\ddelta_\diamond$ is obtained by taking the mates of the naturality squares along the right-pointing edges and pasting them with the remaining squares.

	We will make use of this description as it will allow us to argue that certain morphisms $H$ become invertible in $\BFree^{(1,0)}_{R,S}(\Delta_{/n})\times_\Delta\{n\}$ for purely diagrammatic reasons often in the form of verifying adjointability of commutative squares.  In particular, \cref{BCeasy,rem:BCeasy} often make these adjointability checks immediate. 
\end{remark}

\begin{notation}\label{not:Qnn}
	We denote by $Q(n,n)^{(k)}_{\id}$ the $k$-th stage of the filtered colimit defining $Q(n,n)_{\id}$. Its objects are zig-zags
	\[
		X_0 \xleftarrow{\ \ell_1\ } X_1 \xrightarrow{\ \ell_2\ } X_2
		\xleftarrow{\ \ell_3\ } X_3 \xrightarrow{\ \ell_4\ } \cdots X_{k+1}
	\]
	in $\Delta$ -- so $\ell_i \colon X_i \to X_{i-1}$ for $i$ odd and $\ell_i \colon X_{i-1} \to X_i$ for $i$ even -- subject to:
	\begin{itemize}
		\item $X_0=X_{k+1}=[n]$;
		\item each even-indexed $\ell_i$ lies in $\cL$, and we write $\ell_{i}^{\ra}$ for its right adjoint;
		\item setting $\varphi_i \coloneqq \ell_i$ for $i$ odd and $\varphi_i \coloneqq \ell_i^{\ra}$ for $i$ even, so that each $\varphi_i \colon X_i \to X_{i-1}$, we have that 
		\[
			\varphi_1 \circ \varphi_2 \circ \cdots \circ \varphi_{k+1}=\id_{[n]}
		\]
		holds.
	\end{itemize}
\end{notation}

\begin{definition}
	For $X \in Q(n,n)^{(k)}_{\id}$ we let $c_r= \varphi_1 \circ \cdots \circ \varphi_r \colon X_r \to [n]$ and dually $a_r=\varphi_{r+1}\circ \cdots \circ \varphi_{k+1} \colon [n] \to X_r$. Note that by construction we have that $c_r \circ a_r = \id_{[n]}$. We set the convention $c_0=a_0=\id_{[n]}$.
\end{definition}

\begin{definition}
	Let $\rho^k_n \colon L(n,n)_{\id}^{(k)} \hookrightarrow Q(n,n)_{\id}^{(k)}$ denote the full sub-category on those objects $X$ such that for every $0 \leq r \leq k+1$ we have that $c_r \dashv a_r$, or equivalently, such that $\id_{X_{r}} \leq a_r \circ c_r$.
\end{definition}

\begin{lemma}\label{lem:restrict}
Let $X \in Q(n,n)^{(k)}_{\id}$ and denote by $M(X_r)=\{x \in X_r | \enspace x \leq a_rc_r(x)\}$. Then:
\begin{enumerate}
  \item[(a)] if $r$ is odd, $\ell_r$ restricts to a map $M(X_r) \to M(X_{r-1})$;
  \item[(b)] if $r$ is even, $\ell_r$ restricts to a map $M(X_{r-1}) \to M(X_r)$;
  \item[(c)] for every $r$, $\varphi_r$ restricts to a map $M(X_r) \to M(X_{r-1})$;
  \item[(d)] if $r$ is even, the map induced by $\ell_r$ in (b) lies in $\cL$.
\end{enumerate}
\end{lemma}
\begin{proof}
Throughout we use $c_r = c_{r-1}\varphi_r$ and $a_{r-1} = \varphi_r a_r$, which
hold for every $r$ by construction. We proceed case by case
\begin{itemize}
	\item[$(a)$] Let $x \in M(X_r)$, so $x \le a_rc_r(x)$. Since $r$ is odd,
$\varphi_r = \ell_r$, and applying the monotone map $\ell_r$ gives
\[
  \ell_r(x) \;\le\; \ell_r a_r c_r(x) \;=\; a_{r-1} c_{r-1}\ell_r(x),
\]
which is precisely the statement $\ell_r(x) \in M(X_{r-1})$.
\item[$(b)$] Let $x \in M(X_{r-1})$, so $x \le a_{r-1}c_{r-1}(x)$. Since $r$ is even, $\varphi_r = \ell_r^{\ra}$, $a_{r-1} = \ell_r^{\ra}a_r$ and
$c_r = c_{r-1}\ell_r^{\ra}$. Therefore
\[
	\begin{aligned}
		\ell_r(x)
		&\leq \ell_r a_{r-1}c_{r-1}(x) \\
		&= \ell_r\ell_r^{\ra}a_rc_{r-1}(x) \\
		&\leq a_rc_{r-1}(x)
		 \leq a_rc_{r-1}\ell_r^{\ra}\ell_r(x) \\
		&= a_rc_r\ell_r(x),
	\end{aligned}
\]
where the second inequality is the counit $\ell_r\ell_r^{\ra} \le \id$ and the
third is the unit $\id \le \ell_r^{\ra}\ell_r$. So $\ell_r(x) \in M(X_r)$.
\item[$(c)$] The case of $r$ odd follows from $(a)$. If $r$ is even we note that for $x \in M(X_r)$ $\ell_r^{\ra}(x) \leq \ell_r^{\ra}a_rc_r(x)$ and since
   \[
	a_{r-1}c_{r-1}\ell_r^{\ra}(x)=a_{r-1} c_r(x)=\ell_r^{\ra} a_r c_r(x),
   \]
   the conclusion holds.
\item[$(d)$] The previous point shows that the adjunction $\ell_r \dashv \ell_r^{\ra}$ restricts to $M(X_{r-1})$ and $M(X_r)$ and so the assertion follows from \cref{rem:Lintrinsinc}.\qedhere
\end{itemize}
\end{proof}

\begin{proposition}\label{prop:Lcoreflect}
	There exists an adjunction $\rho^k_n \colon L(n,n)_{\id}^{(k)} \llra Q(n,n)_{\id}^{(k)} \colon M$ for every $k,n \geq 0$.
\end{proposition}
\begin{proof}
	Let $X \in Q(n,n)_{\id}^{(k)}$. \cref{lem:restrict} shows that the various $M(X_r)$ assemble into an object $M(X)\in Q(n,n)^{(k)}$. Since $a_{r}c_ra_r=a_r$, it follows that each $a_r$ always factors through $M(X_r)$ which shows that $M(X)\in Q(n,n)^{(k)}_{\id}$. It is immediate from our definitions that $M(X)$ factors through $L(n,n)^{(k)}_\id$. To conclude the proof we will show that, given $Y \in L(n,n)_{\id}^{(k)}$, each morphism $f \colon Y \to X$ factors uniquely through $M(X)$. Passing to mates it is clear that $f_{r-1} \varphi_{r}^Y \leq \varphi_r^X f_r$. In particular, this shows that
	\[
		f_r(y) \leq f_r a_r^{Y}c_r^{Y}(y) \leq a_r^{X}c_r^{X}f_r(y). \qedhere
	\]
	
\end{proof}

\begin{proposition}\label{prop:Qcontract}
	For every $k \geq 1$, the category $L(n,n)_{\id}^{(k)}$ admits a terminal object.
\end{proposition}
\begin{proof}
	
	Let $I_n^{(k)}$ denote the identity zig-zag on $\id_{[n]}$, and fix $X \in L(n,n)_{\id}^{(k)}$. We claim that the morphisms $c_r$, $0\leq r \leq k+1$, assemble into a morphism $c \colon X \to I_n^{(k)}$. Recall that $c_r=c_{r-1}\varphi_r$ for every $r$. If $r$ is odd, then $\varphi_r=\ell_r$ and the required compatibility $c_{r-1}\ell_r=c_r$ is immediate. Suppose instead that $r$ is even, so we must show $c_r\ell_r=c_{r-1}$. From $c_r=c_{r-1}\varphi_r$ we get $c_r\ell_r=c_{r-1}\varphi_r\ell_r$, and the adjunction $\ell_r \dashv \varphi_r$ gives $c_{r-1} \leq c_r \ell_r$. For the reverse inequality, we use that $X \in L(n,n)_{\id}^{(k)}$, i.e.\ that $x \leq a_{r-1}c_{r-1}(x)$ for every $x \in X_{r-1}$:
	\[
		\begin{aligned}
			c_r \ell_r(x) = c_{r-1}\varphi_{r} \ell_r(x)
			&\leq c_{r-1}\varphi_{r} \ell_r a_{r-1}c_{r-1}(x) \\
			&= c_{r-1}\varphi_{r} \ell_r \varphi_r a_{r}c_{r-1}(x) \\
			&= c_{r-1}\varphi_r a_r c_{r-1}(x) \\
			&= c_{r-1}a_{r-1}c_{r-1}(x) \\
			&= c_{r-1}(x).
		\end{aligned}
	\]
	Combined with $c_{r-1} \leq c_r\ell_r$ above, this gives $c_r\ell_r=c_{r-1}$, as desired.

	We now show that $c$ is the unique such morphism: let $h \colon X \to I_n^{(k)}$ be another morphism, with components $h_r$, $0 \leq r \leq k+1$. First, we verify that $c_r \leq h_r $. To this end it will be important to recall the following inequality
	\begin{equation}
		h_{r-1}\varphi_r \leq h_r
	\end{equation}
	which follows easily from the calculus of mates. In particular we see working inductively with base case $c_0=h_0=\id_{[n]}$ that 
	\[
		c_{r}= c_{r-1}\varphi_r \leq h_{r-1}\varphi_r \leq h_r
	\]
	where the first inequality uses the induction hypothesis.
	To see that $ h_r \leq c_r$, we will use induction starting from $r=k+1$ where we have $h_{k+1}=c_{k+1}=\id_{[n]}$. We now use that $X=M(X)$ to obtain
	\[
		h_{r-1} \leq h_{r-1} a_{r-1} c_{r-1} = h_{r-1} \varphi_r a_{r}c_{r-1} \leq  h_r a_{r} c_{r-1} \leq c_r a_{r}c_{r-1}=c_{r-1}.
	\]
	The result now follows.
\end{proof}

\begin{lemma}\label{lem:deltaisos}
	Every morphism in $Q(n,n)_{\id}$ becomes invertible in $\BFree^{(1,0)}_{R,S}(\Delta_{/n})\times_\Delta \{n\}$ (cf. \cref{def:representableRS}).
\end{lemma}
\begin{proof}
	We will first show that the map $M(X)\to X$ from \cref{prop:Lcoreflect} becomes invertible. Let $s$ be even. We consider the commutative square of the form
	\[
		\begin{tikzcd}
			M(X_{s-1}) \arrow[r] \arrow[d] &  M(X_{s}) \arrow[d] \\
			X_{s-1} \arrow[r,"\ell_s"] & X_{s}
		\end{tikzcd}
	\] 
	and claim that this square is sent to a horizontally adjointable square by any functor $F \colon \Delta \to \cX$ extending to $\ddelta_{\diamond}$. Note that by the discussion in \cref{rem:univadjointable} this will imply that $M(X)\to X$ maps to an invertible morphism.

	Let us factor $\ell_s\colon X_{s-1}\xrightarrow{\pi_s}Y \xrightarrow{\delta_s} X_{s}$ as a surjection followed by an injection where both maps preserve the bottom element. We define $c_{Y}=c_{s-1} \circ \pi_{s}^{\ra}$ and $a_{Y}=\delta_s^{\ra}\circ a_s$ and consider $M(Y)=\{y \in Y \mid y \leq a_{Y}c_{Y}(y)\}$. We claim we have a commutative diagram
	\[
		\begin{tikzcd}
			M(X_{s-1}) \arrow[d] \arrow[r] & M(Y) \arrow[d] \arrow[r] & M(X_s) \arrow[d] \\
			X_{s-1} \arrow[r,"\pi_s"] & Y \arrow[r,"\delta_s"] & X_s
		\end{tikzcd}
	\]
	To see this we consider $x \in M(X_{s-1})$ and compute
	\[
		\begin{aligned}
			\pi_s a_{s-1}c_{s-1}(x)
			&= \pi_s \pi_s^{\ra} \delta_s^{\ra}a_s c_{s-1}(x) \\
			&= \delta_s^{\ra}a_s c_{s-1}(x) \\
			&\leq \delta_s^{\ra}a_s c_{s-1}\pi_s^{\ra}\pi_s(x) \\
			&= a_Yc_Y\pi_s(x),
		\end{aligned}
	\]
	which shows that $\pi_s(x)\in M(Y)$. For the remaining square we pick $y \in M(Y)$ and compute
	\[
		\delta_s a_Y c_Y(y)=\delta_s \delta_s^{\ra}a_s c_{s-1}\pi_s^{\ra}(y) \leq a_s c_{s-1}\pi_s^{\ra}(y)=a_s c_{s-1}\pi_s^{\ra}\delta_s^{\ra}\delta_s(y)=a_s c_s \delta_s(y)
	\]
	showing that $\delta_s(y) \leq a_s c_s \delta_s(y)$ and thus $\delta_s(y) \in M(X_s)$. We will now show that $\pi_s^{\ra}$ and $\delta_s^{\ra}$ also factor thus showing that both squares are horizontally adjointable. For this we verify that
	\[
		a_{s-1}c_{s-1}\pi_s^{\ra}(y)=\pi_s^{\ra} a_Y c_Y(y), \enspace  a_Y c_Y \delta_s^{\ra}(x)=\delta_s^{\ra}a_s c_s(x),
	\]
	which shows that the right adjoints factor. It follows from \cref{BCeasy,rem:BCeasy} that any functor extending to $\ddelta_{\diamond}$ must send each of the squares above to an adjointable square. This implies the analogous claim for the composite square.

	We may now invoke \cref{prop:Qcontract} to reduce the statement to showing that the terminal map $Y \to I_n^{(k)}$ becomes invertible where $Y \in L(n,n)_{\id}^{(k)}$. Again, this amounts to verifying that the diagram below
	\[
		\begin{tikzcd}
			Y_{r-1} \arrow[r,"\ell_{r}"] \arrow[d,"c_{r-1}"] & Y_{r} \arrow[d,"c_{r}"] \\
			{[n]} \arrow[r,"\id_{[n]}"] & {[n]}
		\end{tikzcd}
	\] 
	is universally horizontally adjointable for functors extending to $\ddelta_\diamond$ whenever $r$ is even. First let us observe that any functor extending to $\ddelta_{\diamond}$ must send the $c_s$'s to right adjoints and consequently we fix the notation $b_{r}\dashv c_r$. Moreover, we note that since $c_r=c_{r-1}\ell^{\ra}_r$ it follows that $b_{r}=\ell_r b_{r-1}$. By \cref{rem:BCeasy} our diagram is thus universally vertically adjointable. Taking this adjointed square and passing to right adjoints of all morphisms yields another commutative square which we identify with the horizontal mate transformation of the functor above. We conclude that the original square is thus universally horizontally adjointable.
\end{proof}
\medskip
\begin{proof}[Proof of \cref{thm:2simplex}]
 We claim it suffices to verify the following:
 \begin{itemize}
	\item[$(\star)$] The functor $Q(n,n)_{\id} \to Q(n,n)$ descends to a fully faithful functor 
	\[
		|Q(n,n)_{\id}| \to \ddelta_{\diamond}(n,n).
	\]
 \end{itemize}
 Indeed, by combining \cref{prop:Lcoreflect} and \cref{prop:Qcontract} we find that $|Q(n,n)_{\id}| \simeq \ast$. The claim implies that $\Map_{\ddelta_{\diamond}(n,n)}(\id_{[n]},\id_{[n]})\simeq \ast$ and the conclusion of the theorem thus follows from \cref{lem:keyreduction2simplex}.

Write $\cT_n \subset Q(n,n)$ for the class of morphisms that become invertible in $\ddelta_{\diamond}(n,n)=\BFree^{(1,0)}_{R,S}(\Delta_{/n})\times_\Delta \{n\}$ so that $Q(n,n)[\cT_n^{-1}]=\ddelta_{\diamond}(n,n)$.

	Let $p\colon Q(n,n)\to\ddelta(n,n)$ denote the canonical functor. If a morphism belongs to $\cT_n$, then its image under $p$ is invertible and hence, since $\ddelta(n,n)$ is a poset, an identity. Now fix $\alpha_0,\alpha_1\in Q(n,n)_{\id}$. By the hammock description of localization (cf. \cref{cor:dwyerkanhammock}), the mapping space $\Map_{\ddelta_{\diamond}(n,n)}(\alpha_0,\alpha_1)$ is computed by zig-zags
	\[
		\alpha_1 \xleftarrow{} \beta_0 \xrightarrow{} \beta_1 \xleftarrow{} \cdots \xrightarrow{}\alpha_0,
	\]
	whose right-pointing morphisms belong to $\cT_n$. Applying $p$ gives a chain
	\[
		\id_{[n]}\geq p(\beta_0)=p(\beta_1)\geq\cdots\geq\id_{[n]},
	\]
	so every $p(\beta_i)$ is equal to $\id_{[n]}$. Thus the entire zig-zag lies in $Q(n,n)_{\id}$. Conversely, by \cref{lem:deltaisos}, every morphism in $Q(n,n)_{\id}$ belongs to $\cT_n$. The inclusion therefore identifies the categories of zig-zags computing the two mapping spaces, and the claim follows.
\end{proof}

\begin{definition}
	Let $\ddelta_+$ be the sub-$(\infty,2)$-category of $\Cat_{(\infty,1)}$ on the finite linearly ordered sets together with the empty category. We refer to $\ddelta_+$ as the \emph{extended simplex} 2-category and denote its underlying 1-category as $\Delta_+$.
\end{definition}

\begin{corollary}\label{cor:extendeddelta}
	Let $R,S$ be the families of sections and retractions from \cref{def:representingdelta}, regarded in $\Delta_+$.
	Then we have an equivalence of $(\infty,2)$-categories $\Delta_{+}\soft{\cL}_{R,S} \simeq \ddelta_+$.
\end{corollary}
\begin{proof}
	Write $\Delta_{+}\soft{\cL}_{R,S}\coloneq\ddelta_{+,\diamond}$ and observe that the canonical map $\ddelta_{+,\diamond} \to \ddelta_{+}$ is essentially surjective. Invoking \cref{thm:2simplex} and comparing the description of the mapping $(\infty,1)$-categories in $\ddelta$ and $\ddelta_{+,\diamond}$ arising from \cref{thm:mappingcatform}, we find that the induced functor $\ddelta \to \ddelta_{+,\diamond}$ is fully faithful. In particular for mapping $(\infty,1)$-categories not involving the initial object $\emptyset$, we see that $\ddelta_{+,\diamond}([m],[n]) \to \ddelta_{+}([m],[n])$ is an equivalence. Since $\ddelta_{+}([m],\emptyset)=\emptyset$, it will be enough to show that $\emptyset \in \ddelta_{+,\diamond}$ is an initial object in the sense that $\ddelta_{+,\diamond}(\emptyset, [n])\simeq [0]$.
	
	Comparing universal properties one obtains a pushout diagram
	\[
		\begin{tikzcd}
			\Delta \arrow[r] \arrow[d] & \ddelta \arrow[d] \\
			\Delta_+ \arrow[r] & \ddelta_{+,\diamond}
			\arrow[ul, phantom, very near start, "\ulcorner"]
		\end{tikzcd}
	\]
	 and since strong $(0,1)$-cofinal functors (\cref{def:strongcofinal}) are stable under cobase change by \cref{prop:cofcobase}, it will be enough to show that the map $\Delta \to \ddelta$ is strongly $(0,1)$-cofinal. By \cref{thm:2simplex} this map is obtained as an iterated pushout along maps $\Ret \to \Adjffr$, $\Ret \to \Adjffl$ so we further reduce the claim to \cref{prop:adjffcofinal}.
\end{proof}

\bibliographystyle{amsalpha}
\bibliography{references}

\end{document}